\documentclass[a4paper,12pt]{amsart}

\usepackage{amssymb}
\usepackage{amsmath}
\usepackage{amsthm}

\usepackage[colorlinks]{hyperref}
\usepackage{hyperref}
\renewcommand\eqref[1]{(\ref{#1})} 

\allowdisplaybreaks \numberwithin{equation}{section}

\theoremstyle{plain}
\newtheorem{theorem}{Theorem}[section]
\newtheorem{prop}[theorem]{Proposition}
\newtheorem{corollary}[theorem]{Corollary}
\newtheorem{lemma}[theorem]{Lemma}
\newtheorem{assump}[theorem]{Assumption}

\theoremstyle{definition}
\newtheorem{defn}[theorem]{Definition}
\newtheorem{rem}[theorem]{Remark}

\def\ind{{\mathcal I}}

\newcommand{\omp}{(\overline{\Omega})}
\newcommand{\efel}{\mathcal{F}_L}
\newcommand{\efela}{\mathcal{F}_{L^*}}

\newcounter{quotecount}

\begin{document}

\title[Nonharmonic analysis without WZ condition]
{Nonharmonic analysis of boundary value problems without WZ condition}

\author[Michael Ruzhansky]{Michael Ruzhansky}
\address{
  Michael Ruzhansky:
  \endgraf
  Department of Mathematics
  \endgraf
  Imperial College London
  \endgraf
  180 Queen's Gate, London, SW7 2AZ
  \endgraf
  United Kingdom
  \endgraf
  {\it E-mail address} {\rm m.ruzhansky@imperial.ac.uk}
  }
\author[Niyaz Tokmagambetov]{Niyaz Tokmagambetov}
\address{
  Niyaz Tokmagambetov:
  \endgraf
    al--Farabi Kazakh National University
  \endgraf
  71 al--Farabi ave., Almaty, 050040
  \endgraf
  Kazakhstan,
  \endgraf
   and
  \endgraf
    Department of Mathematics
  \endgraf
  Imperial College London
  \endgraf
  180 Queen's Gate, London, SW7 2AZ
  \endgraf
  United Kingdom
  \endgraf
  {\it E-mail address} {\rm n.tokmagambetov@imperial.ac.uk}
 }

\date{\today}

\thanks{The first author was supported in parts by the EPSRC
 grant EP/K039407/1 and by the Leverhulme Grant RPG-2014-02. The second author was supported in parts by the MESRK grant 0773/GF4. No new data was collected or generated during the course of this research.}

\maketitle

\begin{abstract}
In this work we continue our research on nonharmonic analysis of
boundary value problems as initiated in \cite{Ruzhansky-Tokmagambetov:IMRN}.
There, we assumed that the eigenfunctions of the model operator on which the construction is based do not have zeros. In this paper we have weakened this condition extending the applicability of the developed pseudo-differential analysis. Also, we do not assume that the underlying set $\Omega$ is bounded.
\end{abstract}

\section{Introduction}

In \cite{Ruzhansky-Tokmagambetov:IMRN} the authors developed
pseudo-differential calculus in terms of the
`model' densely defined operator L. The main examples are operators in
$\Omega\subset\mathbb R^n$ equipped with (arbitrary) boundary conditions on $\partial\Omega$ for which the global Fourier analysis in terms of its
eigenfunctions was introduced. Such a `model' operator L does
not have to be self-adjoint, so the construction is based on
biorthogonal systems rather than on an orthonormal basis (to take
into account the non-self-adjointness). Also, the operator L
does not have to be elliptic. The `model' operator ${\rm L}$ was considered as a differential operator of
order $m$ with smooth coefficients on an open  bounded set
$\Omega\subset\mathbb R^n$ equipped with some boundary
conditions which one can denote as (BC). In \cite{Ruzhansky-Tokmagambetov:IMRN} one worked with discrete sets of eigenvalues and eigenfunctions indexed by
a countable set, and one developed elements of the symbolic calculus assuming that the system of eigenfunctions is the without zeros in $\Omega$ (so called WZ-system).
We refer to \cite{Ruzhansky-Tokmagambetov:IMRN} for examples and an extensive list of references in this subject.

\smallskip

In this paper we will drop some conditions of  the `model'
operator ${\rm L}$. Let us consider the case when ${\rm L}$ is an
arbitrary operator in $\Omega\subseteq\mathbb R^n$ with the
discrete spectrum and the system of eigenfunctions which is a
Riesz basis in $L^{2}(\Omega)$.

Denote the corresponding countable index set by $\ind$. However, in different problems it may be more convenient to make different
choices for this set, e.g. $\ind=\mathbb N$ or $\mathbb Z$  or $\mathbb Z^k$, etc.
In order to allow different applications we will be denoting it by $\ind$,
and without loss of generality we
will assume that
\begin{equation}\label{EQ:ind}
\ind \textrm{ is a subset of } \mathbb Z^{K} \textrm{ for some } K\geq 1.
\end{equation}
For simplicity, one can think of $\ind=\mathbb Z$ or $\ind=\mathbb N\cup\{0\}$
throughout this paper.
Thus, in this paper we will be always working in the following setting:

\begin{assump}
\label{Assumption_1} Let $\Omega\subseteq\mathbb R^{n}$, $n\geq
1$, be an open set.
Assume that ${\rm L}$ is a densely defined operator with a discrete spectrum
$\{\lambda_{\xi}\in\mathbb C: \, \xi\in\ind\}$ on $L^{2}(\Omega)$, and the system of corresponding eigenfunctions
$\{u_{\xi}: \; \xi\in\ind\}$ is a Riesz basis in $L^{2}(\Omega)$ (i.e.
for every $f\in L^{2}(\Omega)$ there exists a unique series
$\sum_{\xi\in\ind} a_\xi u_\xi(x)$ that converges to $f$ in  $L^{2}(\Omega)$),
where $\ind$ is a countable set as in \eqref{EQ:ind}, and we order
the eigenvalues with the occurring multiplicities in the ascending
order:
\begin{equation}\label{EQ: EVOrder}
 |\lambda_{j}|\leq|\lambda_{k}| \quad\textrm{ for } |j|\leq |k|.
\end{equation}
\end{assump}


We denote by $u_{\xi}$  the eigenfunction of ${\rm L}$
corresponding to the eigenvalue $\lambda_{\xi}$ for each
$\xi\in\ind$, so that
\begin{equation}
\label{SpecPr} {\rm L}u_{\xi}=\lambda_{\xi}u_{\xi} \,\,\,\,\,\,
\textrm{ in }  \Omega,\quad \textrm{ for all } \xi\in\ind.
\end{equation}
The conjugate spectral problem is
\begin{equation}
\label{ConjSpecPr} {\rm
L^{\ast}}v_{\xi}=\overline{\lambda}_{\xi}v_{\xi}\,\,\,\,\,\,
\textrm{ in } \Omega\quad \textrm{ for all } \xi\in\ind.
\end{equation}

Let $\|u_{\xi}\|_{L^{2}}=1$ and $\|v_{\xi}\|_{L^{2}}=1$ for all
$\xi\in\ind.$ Here, we can take biorthogonal systems
$\{u_{\xi}\}_{\xi\in\ind}$ and $\{v_{\xi}\}_{\xi\in\ind}$,
i.e.
\begin{equation}\label{BiorthProp}
(u_{\xi},v_{\eta})_{L^2}=0 \,\,\,\, \hbox{for}
\,\,\,\, \xi\neq\eta, \,\,\,\, \hbox{and} \,\,\,\,
(u_{\xi},v_{\eta})_{L^2}=1 \,\,\,\, \hbox{for} \,\,\,\, \xi=\eta,
\end{equation}
where
$$(f, g)_{L^{2}}:=\int_{\Omega}f(x)\overline{g(x)}dx$$ is the usual
inner product of the
Hilbert space $L^{2}(\Omega)$. From N.K. Bari's work
\cite{bari} it follows that the system $\{u_{\xi}: \,\,\,
\xi\in\ind\}$ is a basis in $L^{2}(\Omega)$ if and only if the
system $\{v_{\xi}: \,\,\, \xi\in\ind\}$ is a basis in
$L^{2}(\Omega)$. Therefore, by Bari \cite{bari}, the system $\{v_{\xi}: \,\,\, \xi\in\ind\}$ is also
a basis in $L^{2}(\Omega)$.
Also, Assumption \ref{Assumption_1}  will imply that the spaces
$C^\infty_{{\rm L}}(\overline\Omega)$ and $C^\infty_{{\rm L}^*}(\overline\Omega)$
of test functions introduced in Subsection \ref{SEC:TD}
are dense in $L^2(\Omega)$.


\medskip
Define the weight
\begin{equation}\label{EQ:angle}
\langle\xi\rangle:=(1+|\lambda_{\xi}|^2)^{\frac{1}{2m}},
\end{equation}
which will be instrumental in measuring the growth/decay of Fourier coefficients and of
symbols. Here $m>0$ is an arbitrary number that we fix throughout the paper.
For simplicity we can take $m=1$. However, if L is, for example, a differential operator, it is convenient to take $m$ to be equal to its order.

To give the interpretation for $\langle\xi\rangle$ in terms of the operator analysis, we can define the operator
${\rm L}^{\circ}$ by setting its values on the basis $u_{\xi}$ by
\begin{equation}\label{EQ:Lo-def}
{\rm L}^{\circ} u_{\xi}:=\overline{\lambda_{\xi}} u_{\xi},\quad
\textrm{ for all } \xi\in\ind.
\end{equation}
If L is self-adjoint, we have ${\rm L}^{\circ}={\rm L}^*={\rm L}$.
Consequently, we can informally think of $\langle\xi\rangle$
as of the eigenvalues of the positive (first order) operator
$({\rm I}+{\rm L^\circ\, L})^{\frac{1}{2m}}.$

With a similar definition for $({\rm L}^{*})^{\circ}$, we can observe that
$({\rm L}^{*})^{\circ}=({\rm L}^{\circ})^{*}$.

Simplest examples of non-periodic boundary conditions were considered in  \cite{Kanguzhin_Tokmagambetov_Tulenov} and \cite{Kanguzhin_Tokmagambetov} in the case of
$\Omega=[0,1]$ being the segment. This extends to the non-periodic case the periodic analysis developed in \cite{RT07, Ruzhansjy-Turunen:NFA, RT,Ruzhansky-Turunen-JFAA-torus} on the torus which can be viewed as analysis on $\Omega=[0,1]$ with periodic boundary conditions.
We refer to \cite{Ruzhansky-Tokmagambetov:IMRN} for further examples.

\section{Preliminary}
\label{SEC:TD}

In this section we collect some results on ${\rm L}$--distributions,
${\rm L}$--Fourier transform, Plan\-che\-rel
formula and Sobolev spaces $\mathcal H^{s}_{{\rm L}}(\Omega)$, and
we omit the proofs because they are a straightforward extension of
those in \cite{Ruzhansky-Tokmagambetov:IMRN}.

\subsection{Global distributions generated by the boundary value
problem}

In this subsection we describe the spaces of distributions
generated by ${\rm L}$ and by its
adjoint ${\rm L}^*$ and the related global Fourier analysis. The
more far-reaching aim of this analysis is to establish a version
of the Schwartz kernel theorem for the appearing spaces of
distributions.
We first define the space $C_{{\rm
L}}^{\infty}(\overline{\Omega})$ of test functions.

\begin{defn}\label{TestFunSp}
The space $C_{{\rm L}}^{\infty}(\overline{\Omega}):={\rm Dom}({\rm
L}^{\infty})$ is called the space of test functions for ${\rm L}$.
Here we define
$$
{\rm Dom}({\rm L}^{\infty}):=\bigcap_{k=1}^{\infty}{\rm Dom}({\rm
L}^{k}),
$$
where ${\rm Dom}({\rm L}^{k})$ is the domain of the operator ${\rm
L}^{k}$, in turn defined as
$$
{\rm Dom}({\rm L}^{k}):=\{f\in L^{2}(\Omega): \,\,\, {\rm
L}^{j}f\in {\rm Dom}({\rm L}), \,\,\, j=0, \,1, \, 2, \ldots,
k-1\}.
$$
The Fr\'echet topology of $C_{{\rm
L}}^{\infty}(\overline{\Omega})$ is given by the family of norms
\begin{equation}\label{EQ:L-top}
\|\varphi\|_{C^{k}_{{\rm L}}}:=\max_{j\leq k}
\|{\rm L}^{j}\varphi\|_{L^2(\Omega)}, \quad k\in\mathbb N_0,
\; \varphi\in C_{{\rm L}}^{\infty}(\overline{\Omega}).
\end{equation}

Analogously to the ${\rm L}$-case, we introduce the space $C_{{\rm
L^{\ast}}}^{\infty}(\overline{\Omega})$ corresponding to the adjoint operator ${\rm L}_\Omega^*$ by
$$
C_{{\rm L^{\ast}}}^{\infty}(\overline{\Omega}):=
{\rm Dom}(({\rm L^{\ast}})^{\infty})=\bigcap_{k=1}^{\infty}{\rm
Dom}(({\rm L^{\ast}})^{k}),
$$
where ${\rm Dom}(({\rm L^{\ast}})^{k})$ is the domain of the
operator $({\rm L^{\ast}})^{k}$,
$$
{\rm Dom}(({\rm L^{\ast}})^{k}):=\{f\in L^{2}(\Omega): \,\,\, ({\rm
L^{\ast}})^{j}f\in {\rm Dom}({\rm L^{\ast}}), \,\,\, j=0, \ldots, k-1\},
$$
which satisfy the adjoint boundary conditions corresponding to the operator
${\rm L}_\Omega^*$. The Fr\'echet topology of $C_{{\rm L}^*}^{\infty}(\overline{\Omega})$ is given by the family of norms
\begin{equation}\label{EQ:L-top-adj}
\|\psi\|_{C^{k}_{{\rm L}^*}}:=\max_{j\leq k}
\|({\rm L}^*)^{j}\psi\|_{L^2(\Omega)}, \quad k\in\mathbb N_0,
\; \psi\in C_{{\rm L}^*}^{\infty}(\overline{\Omega}).
\end{equation}

Since we have $u_\xi\in C^\infty_{{\rm L}}(\overline\Omega)$ and
$v_\xi\in C^\infty_{{\rm L}^*}(\overline\Omega)$ for all $\xi\in\ind$, we observe that
Assumption \ref{Assumption_1} implies that the spaces
$C^\infty_{{\rm L}}(\overline\Omega)$ and $C^\infty_{{\rm L}^*}(\overline\Omega)$
are dense in $L^2(\Omega)$.
\end{defn}

We note that if ${\rm L}$ is self-adjoint, i.e. if ${\rm L}^*={\rm
L}$ with the equality of domains, then $C_{{\rm
L^{\ast}}}^{\infty}(\overline{\Omega})=C_{{\rm
L}}^{\infty}(\overline{\Omega}).$

In general, for functions $f\in C_{{\rm L}}^{\infty}(\overline{\Omega})$ and
$g\in C_{{\rm L}^*}^{\infty}(\overline{\Omega})$, the $L^2$-duality makes sense in view
of the formula
\begin{equation}\label{EQ:duality}
({\rm L}f, g)_{L^2(\Omega)}=(f,{\rm L}^*g)_{L^2(\Omega)}.
\end{equation}
Therefore, in view of the formula \eqref{EQ:duality},
it makes sense to define the distributions $\mathcal D'_{{\rm L}}(\Omega)$
as the space which is dual to $C_{{\rm L}^*}^{\infty}(\overline{\Omega})$.

\begin{defn}\label{DistrSp}
The space $$\mathcal D'_{{\rm
L}}(\Omega):=\mathcal L(C_{{\rm L}^*}^{\infty}(\overline{\Omega}),
\mathbb C)$$ of linear continuous functionals on
$C_{{\rm L}^*}^{\infty}(\overline{\Omega})$ is called the space of
${\rm L}$-distributions.
We can understand the continuity here either in terms of the topology
\eqref{EQ:L-top-adj} or in terms of sequences, see
Proposition \ref{TH: UniBdd}.
For
$w\in\mathcal D'_{{\rm L}}(\Omega)$ and $\varphi\in C_{{\rm L}^*}^{\infty}(\overline{\Omega})$,
we shall write
$$
w(\varphi)=\langle w, \varphi\rangle.
$$
For any $\psi\in C_{{\rm L}}^{\infty}(\overline{\Omega})$,
$$
C_{{\rm L}^*}^{\infty}(\overline{\Omega})\ni \varphi\mapsto\int_{\Omega}{\psi(x)} \, \varphi(x)\, dx
$$
is an ${\rm L}$-distribution, which gives an embedding $\psi\in
C_{{\rm L}}^{\infty}(\overline{\Omega})\hookrightarrow\mathcal D'_{{\rm
L}}(\Omega)$.
We note that in the distributional notation formula \eqref{EQ:duality} becomes
\begin{equation}\label{EQ:duality-dist}
\langle{\rm L}\psi, \varphi\rangle=\langle \psi,\overline{{\rm L}^* \overline{\varphi}}\rangle.
\end{equation}
\end{defn}

With the topology on $C_{{\rm L}}^{\infty}(\overline{\Omega})$
defined by \eqref{EQ:L-top},
the space $$\mathcal
D'_{{\rm L^{\ast}}}(\Omega):=\mathcal L(C_{{\rm L}}^{\infty}(\overline{\Omega}), \mathbb C)$$
of linear continuous functionals on $C_{{\rm L}}^{\infty}(\overline{\Omega})$
is called the
space of ${\rm L^{\ast}}$-distributions.

\begin{prop}\label{TH: UniBdd}
A linear functional $w$ on
$C_{{\rm L}^*}^{\infty}(\overline{\Omega})$ belongs to $\mathcal D'_{{\rm
L}}(\Omega)$ if and only if there exists a constant $c>0$ and a
number $k\in\mathbb N_0$ with the property
\begin{equation}
\label{EQ: UnifBdd-s1} |w(\varphi)|\leq
c \|\varphi\|_{C^{k}_{{\rm L}^*}} \quad \textrm{ for all } \; \varphi\in C_{{\rm
L}^*}^{\infty}(\overline{\Omega}).
\end{equation}
\end{prop}

The space $\mathcal D'_{{\rm L}}(\Omega)$ has many similarities with the
usual spaces of distributions. For example, suppose that for a linear continuous operator
$D:C_{{\rm L}}^{\infty}(\overline{\Omega})\to C_{{\rm L}}^{\infty}(\overline{\Omega})$
its adjoint $D^*$
preserves the domain of ${\rm L}^*$
and is continuous on the space
$C_{{\rm L}^*}^{\infty}(\overline{\Omega})$, i.e.
that the operator
$D^*:C_{{\rm L}^*}^{\infty}(\overline{\Omega})\to C_{{\rm L}^*}^{\infty}(\overline{\Omega})$
is continuous.
Then we can extend $D$ to $\mathcal D'_{{\rm L}}(\Omega)$ by
$$
\langle Dw,{\varphi}\rangle := \langle w, \overline{D^* \overline{\varphi}}\rangle \quad
(w\in \mathcal D'_{{\rm L}}(\Omega),\;  \varphi\in C_{{\rm
L}^*}^{\infty}(\overline{\Omega})).
$$
This extends \eqref{EQ:duality-dist} from L to other operators.
The convergence in the linear space
$\mathcal D'_{{\rm L}}(\Omega)$ is the usual weak convergence with respect to
the space $C_{{\rm L}^*}^{\infty}(\overline{\Omega})$.
The following principle of uniform boundedness is based on the
Banach--Steinhaus Theorem applied to the Fr\'echet space $C_{{\rm
L}^*}^{\infty}(\overline{\Omega})$.

\begin{lemma} \label{LEM: UniformBoundedness}
Let $\{w_{j}\}_{j\in\mathbb N}$ be a sequence in $\mathcal
D'_{{\rm L}}(\Omega)$ with the property that for every $\varphi\in
C_{{\rm L}^*}^{\infty}(\overline{\Omega})$, the sequence
$\{w_{j}(\varphi)\}_{j\in\mathbb N}$ in $\mathbb C$ is bounded.
Then there exist constants $c>0$ and $k\in\mathbb N_0$ such that
\begin{equation}
\label{EQ: UniformBoundedness} |w_{j}(\varphi)|\leq c
\|\varphi\|_{C^{k}_{{\rm L}^*}} \quad \textrm{ for all } \; j\in\mathbb N, \,\,
\varphi\in C_{{\rm L}^*}^{\infty}(\overline{\Omega}).
\end{equation}
\end{lemma}

The lemma above leads to the following property of completeness of
the space of ${\rm L}$-distributions.

\begin{theorem} \label{TH: Com-nessDistr}
Let $\{w_{j}\}_{j\in\mathbb N}$ be a
sequence in $\mathcal D'_{{\rm L}}(\Omega)$ with the property that
for every $\varphi\in C_{{\rm L}^*}^{\infty}(\overline{\Omega})$ the
sequence $\{w_{j}(\varphi)\}_{j\in\mathbb N}$ converges in
$\mathbb C$ as $j\rightarrow\infty$. Denote the limit by
$w(\varphi)$.

{\rm (i)} Then $w:\varphi\mapsto w(\varphi)$ defines an ${\rm
L}$-distribution on $\Omega$. Furthermore,
$$
\lim_{j\rightarrow\infty}w_{j}=w \,\,\,\,\,\,\,\, \hbox{in}
\,\,\,\,\,\,\, \mathcal D'_{{\rm L}}(\Omega).
$$

{\rm (ii)} If $\varphi_{j}\rightarrow\varphi$ in $\in C_{{\rm
L}^*}^{\infty}(\overline{\Omega})$, then
$$
\lim_{j\rightarrow\infty}w_{j}(\varphi_{j})=w(\varphi) \,\,\,\,\,\,\,\,
\hbox{in} \,\,\,\,\,\,\, \mathbb C.
$$
\end{theorem}

Similarly to the previous case, we have analogues of
Proposition \ref{TH: UniBdd} and Theorem \ref{TH: Com-nessDistr}
for ${\rm L^{\ast}}$-distributions.

\subsection{${\rm L}$-Fourier transform} \label{SEC:FT}

In this subsection we define the ${\rm L}$-Fourier transform
generated by our operator ${\rm L}$ and its main properties. The
main difference between the self-adjoint and non-self-adjoint
problems ${\rm L}$ is that in the latter case we have to make sure
that we use the right functions from the available biorthogonal
families of $u_\xi$ and $v_\xi$. We start by defining the spaces
that we will obtain on the Fourier transform side.


\medskip
Let $\mathcal S(\ind)$ denote the space of rapidly decaying
functions $\varphi:\ind\rightarrow\mathbb C$. That is,
$\varphi\in\mathcal S(\ind)$ if for any $M<\infty$ there
exists a constant $C_{\varphi, M}$ such that
$$
|\varphi(\xi)|\leq C_{\varphi, M}\langle\xi\rangle^{-M}
$$
holds for all $\xi\in\ind$.
Here $\langle\xi\rangle$ is already adapted to our case
since it is defined by \eqref{EQ:angle}.

The topology on $\mathcal
S(\ind)$ is given by the seminorms $p_{k}$, where
$k\in\mathbb N_{0}$ and $$p_{k}(\varphi):=\sup_{\xi\in\ind}\langle\xi\rangle^{k}|\varphi(\xi)|.$$
Continuous linear functionals on $\mathcal S(\ind)$ are of
the form
$$
\varphi\mapsto\langle u, \varphi\rangle:=\sum_{\xi\in\ind}u(\xi)\varphi(\xi),
$$
where functions $u:\ind \rightarrow \mathbb C$ grow at most
polynomially at infinity, i.e. there exist constants $M<\infty$
and $C_{u, M}$ such that
$$
|u(\xi)|\leq C_{u, M}\langle\xi\rangle^{M}
$$
holds for all $\xi\in\ind$. Such distributions $u:\ind
\rightarrow \mathbb C$ form the space of distributions which we denote by
$\mathcal S'(\ind)$.
We now define the L-Fourier transform on $C_{{\rm L}}^{\infty}(\overline{\Omega})$.

\begin{defn} \label{FT}
We define the ${\rm L}$-Fourier transform
$$
(\mathcal F_{{\rm L}}f)(\xi)=(f\mapsto\widehat{f}):
C_{{\rm L}}^{\infty}(\overline{\Omega})\rightarrow\mathcal S(\ind)
$$
by
\begin{equation}
\label{FourierTr}
\widehat{f}(\xi):=(\mathcal F_{{\rm L}}f)(\xi)=\int_{\Omega}f(x)\overline{v_{\xi}(x)}dx.
\end{equation}
Analogously, we define the ${\rm L}^{\ast}$-Fourier
transform
$$
(\mathcal F_{{\rm L}^{\ast}}f)(\xi)=(f\mapsto\widehat{f}_{\ast}):
C_{{\rm L}^{\ast}}^{\infty}(\overline{\Omega})\rightarrow\mathcal
S(\ind)
$$
by
\begin{equation}\label{ConjFourierTr}
\widehat{f}_{\ast}(\xi):=(\mathcal F_{{\rm
L}^{\ast}}f)(\xi)=\int_{\Omega}f(x)\overline{u_{\xi}(x)}dx.
\end{equation}
\end{defn}

The expressions \eqref{FourierTr} and \eqref{ConjFourierTr}
are well-defined by the Cauchy-Schwarz inequality, for example,
\begin{equation}
\label{EQ: Ineq1}
|\widehat{f}(\xi)|=\left|\int_{\Omega}f(x)\overline{v_{\xi}(x)}dx\right|\leq\|f\|_{L^{2}}
\|v_{\xi}\|_{L^{2}}=\|f\|_{L^{2}}<\infty.
\end{equation}
Moreover, we have

\begin{prop}\label{LEM: FTinS}
The ${\rm L}$-Fourier transform
$\mathcal F_{{\rm L}}$ is a bijective homeomorphism from $C_{{\rm
L}}^{\infty}(\overline{\Omega})$ to $\mathcal S(\ind)$.
Its inverse  $$\mathcal F_{{\rm L}}^{-1}: \mathcal S(\ind)
\rightarrow C_{{\rm L}}^{\infty}(\overline{\Omega})$$ is given by
\begin{equation}
\label{InvFourierTr} (\mathcal F^{-1}_{{\rm
L}}h)(x)=\sum_{\xi\in\ind}h(\xi)u_{\xi}(x),\quad h\in\mathcal S(\ind),
\end{equation}
so that the Fourier inversion formula becomes
\begin{equation}
\label{InvFourierTr0}
f(x)=\sum_{\xi\in\ind}\widehat{f}(\xi)u_{\xi}(x)
\quad \textrm{ for all } f\in C_{{\rm
L}}^{\infty}(\overline{\Omega}).
\end{equation}
Similarly,  $\mathcal F_{{\rm L}^{\ast}}:C_{{\rm L}^{\ast}}^{\infty}(\overline{\Omega})\to \mathcal S(\ind)$
is a bijective homeomorphism and its inverse
$$\mathcal F_{{\rm L}^{\ast}}^{-1}: \mathcal S(\ind)\rightarrow
C_{{\rm L}^{\ast}}^{\infty}(\overline{\Omega})$$ is given by
\begin{equation}
\label{ConjInvFourierTr} (\mathcal F^{-1}_{{\rm
L}^{\ast}}h)(x):=\sum_{\xi\in\ind}h(\xi)v_{\xi}(x), \quad h\in\mathcal S(\ind),
\end{equation}
so that the conjugate Fourier inversion formula becomes
\begin{equation}
\label{ConjInvFourierTr0} f(x)=\sum_{\xi\in\ind}\widehat{f}_{\ast}(\xi)v_{\xi}(x)\quad \textrm{ for all } f\in C_{{\rm
L^*}}^{\infty}(\overline{\Omega}).
\end{equation}
\end{prop}

By dualising the inverse ${\rm L}$-Fourier
transform $\mathcal F_{{\rm L}}^{-1}: \mathcal S(\ind)
\rightarrow C_{{\rm L}}^{\infty}(\overline{\Omega})$, the
${\rm L}$-Fourier transform extends uniquely to the mapping
$$\mathcal F_{{\rm L}}: \mathcal D'_{{\rm L}}(\Omega)\rightarrow
\mathcal S'(\ind)$$ by the formula
\begin{equation}\label{EQ: FTofDistr}
\langle\mathcal F_{{\rm L}}w, \varphi\rangle:=
\langle w,\overline{\mathcal F_{{\rm L}^*}^{-1}\overline{\varphi}}\rangle,
\quad\textrm{ with } w\in\mathcal D'_{{\rm L}}(\Omega),\; \varphi\in\mathcal
S(\ind).
\end{equation}
It can be readily seen that if $w\in\mathcal D'_{{\rm
L}}(\Omega)$ then $\widehat{w}\in\mathcal S'(\ind)$.
The reason for taking complex conjugates in \eqref{EQ: FTofDistr}
is that, if $w\in C_{{\rm L}}^{\infty}(\overline{\Omega})$, we have
the equality
\begin{multline*}
\langle \widehat{w},\varphi\rangle =
\sum_{\xi\in\ind} \widehat{w}(\xi) \varphi(\xi)=
\sum_{\xi\in\ind} \left( \int_\Omega w(x) \overline{v_\xi(x)}dx\right) \varphi(\xi)\\
=
\int_\Omega w(x) \overline{\left( \sum_{\xi\in\ind} \overline{\varphi(\xi)} v_\xi(x)\right)} dx
=
\int_\Omega w(x) \overline{\left( \mathcal F_{{\rm L}^*}^{-1} \overline{\varphi} \right)} dx
=\langle w,\overline{\mathcal F_{{\rm L}^*}^{-1}\overline{\varphi}}\rangle.
\end{multline*}
Analogously, we have the mapping
$$\mathcal F_{{\rm L}^*}: \mathcal D'_{{\rm L}^*}(\Omega)\rightarrow
\mathcal S'(\ind)$$ defined by the formula
\begin{equation}\label{EQ: FTofDistr2}
\langle\mathcal F_{{\rm L}^*}w, \varphi\rangle:=
\langle w,\overline{\mathcal F_{{\rm L}}^{-1}\overline{\varphi}}\rangle,
\quad\textrm{ with } w\in\mathcal D'_{{\rm L}^*}(\Omega),\; \varphi\in\mathcal
S(\ind).
\end{equation}
It can be also seen that if $w\in\mathcal D'_{{\rm
L}^*}(\Omega)$ then $\widehat{w}\in\mathcal S'(\ind)$.

We note that since systems of  $u_\xi$ and of $v_\xi$ are Riesz bases,
we can also compare $L^2$-norms of functions with sums of squares of Fourier coefficients.
The following statement follows from the work of Bari \cite[Theorem 9]{bari}:

\begin{lemma}\label{LEM: FTl2}
There exist constants $k,K,m,M>0$ such that for every $f\in L^{2}(\Omega)$
we have
$$
m^2\|f\|_{L^{2}}^2 \leq \sum_{\xi\in\ind} |\widehat{f}(\xi)|^2\leq M^2\|f\|_{L^{2}}^2
$$
and
$$
k^2\|f\|_{L^{2}}^2 \leq \sum_{\xi\in\ind} |\widehat{f}_*(\xi)|^2\leq K^2\|f\|_{L^{2}}^2.
$$
\end{lemma}

However, we note that the Plancherel identity can be also achieved in suitably
defined $l^2$-spaces of Fourier coefficients, see Proposition \ref{PlanchId}.

\subsection{Plancherel formula, Sobolev spaces $\mathcal
H^{s}_{{\rm L}}(\Omega)$, and their Fourier images}
\label{SEC:Sobolev}

In this subsection we discuss Sobolev spaces adapted to ${\rm
L}$ and their images under the L-Fourier transform. We
start with the $L^2$-setting, where we can recall inequalities
between $L^2$-norms of functions and sums of squares of their
Fourier coefficients, see Lemma \ref{LEM: FTl2}. However, below we
show that we actually have the Plancherel identity in a suitably
defined space $l^{2}_{{\rm L}}$ and its conjugate $l^{2}_{{\rm
L}^{*}}$.

\medskip
Let us denote by $$l^{2}_{{\rm L}}=l^2({\rm L})$$
the linear space of complex-valued functions $a$
on $\ind$ such that $\mathcal F^{-1}_{{\rm L}}a\in
L^{2}(\Omega)$, i.e. if there exists $f\in L^{2}(\Omega)$ such that $\mathcal F_{{\rm L}}f=a$.
Then the space of sequences $l^{2}_{{\rm L}}$ is a
Hilbert space with the inner product
\begin{equation}\label{EQ: InnerProd SpSeq-s}
(a,\ b)_{l^{2}_{{\rm
L}}}:=\sum_{\xi\in\ind}a(\xi)\ \overline{(\mathcal F_{{\rm
L^{\ast}}}\circ\mathcal F^{-1}_{{\rm L}}b)(\xi)}
\end{equation}
for arbitrary $a,\,b\in l^{2}_{{\rm L}}$.
The reason for this choice of the definition is the following formal calculation:
\begin{align}\label{EQ:PL-prelim} \nonumber
(a,\ b)_{l^{2}_{{\rm L}}}&
=\sum_{\xi\in\ind}a(\xi)\ \overline{(\mathcal F_{{\rm L^{\ast}}}\circ\mathcal F^{-1}_{{\rm L}}b)(\xi)}\\ \nonumber
&=\sum\limits_{\xi\in\ind
}a(\xi)\int_{\Omega}\overline{(\mathcal F^{-1}_{{\rm L}}b)(x)}u_{\xi}(x)dx\\ \nonumber
&=\int_{\Omega}\left[\sum\limits_{\xi\in\ind}a(\xi)u_{\xi}(x)\right]\overline{(\mathcal F^{-1}_{{\rm
L}}b)(x)}dx\\ \nonumber
&=\int_{\Omega}(\mathcal F^{-1}_{{\rm L}}a)(x)\overline{(\mathcal F^{-1}_{{\rm L}}b)(x)} dx\\
&=(\mathcal F^{-1}_{{\rm L}}a,\,\mathcal F^{-1}_{{\rm L}}b)_{L^{2}},
\end{align}
which implies the Hilbert space properties of the space of sequences
$l^{2}_{{\rm L}}$. The norm of $l^{2}_{{\rm L}}$ is then given by the
formula
\begin{equation}\label{EQ:l2norm}
\|a\|_{l^{2}_{{\rm L}}}=\left(\sum_{\xi\in\ind}a(\xi)\
\overline{(\mathcal F_{{\rm L^{\ast}}}\circ\mathcal F^{-1}_{{\rm
L}}a)(\xi)}\right)^{1/2}, \quad \textrm{ for all } \; a\in l^{2}_{{\rm L}}.
\end{equation}
We note that individual terms in this sum may be complex-valued but the
whole sum is real and nonnegative due to formula \eqref{EQ:PL-prelim}.

Analogously, we introduce the
Hilbert space $$l^{2}_{{\rm L^{\ast}}}=l^{2}({\rm L^{\ast}})$$
as the space of functions $a$ on $\ind$
such that $\mathcal F^{-1}_{{\rm L^{\ast}}}a\in L^{2}(\Omega)$,
with the inner product
\begin{equation}
\label{EQ: InnerProd SpSeq-s_2} (a,\ b)_{l^{2}_{{\rm
L^{\ast}}}}:=\sum_{\xi\in\ind}a(\xi)\ \overline{(\mathcal
F_{{\rm L}}\circ\mathcal F^{-1}_{{\rm L^{\ast}}}b)(\xi)}
\end{equation}
for arbitrary $a,\,b\in l^{2}_{{\rm L^{\ast}}}$. The norm of
$l^{2}_{{\rm L^{\ast}}}$ is given by the formula
$$
\|a\|_{l^{2}_{{\rm L^{\ast}}}}=\left(\sum_{\xi\in\ind}a(\xi)\
\overline{(\mathcal F_{{\rm L}}\circ\mathcal F^{-1}_{{\rm
L^{\ast}}}a)(\xi)}\right)^{1/2}
$$
for all $a\in l^{2}_{{\rm L^{\ast}}}$. The spaces of sequences
$l^{2}_{{\rm L}}$ and
$l^{2}_{{\rm L^{\ast}}}$ are thus generated by biorthogonal systems
$\{u_{\xi}\}_{\xi\in\ind}$ and $\{v_{\xi}\}_{\xi\in\ind}$.
The reason for their definition in the above forms becomes clear again
in view of the following Plancherel identity:

\begin{prop} {\rm(Plancherel's identity)}\label{PlanchId}
If $f,\,g\in L^{2}(\Omega)$ then
$\widehat{f},\,\widehat{g}\in l^{2}_{{\rm L}}, \,\,\,
\widehat{f}_{\ast},\, \widehat{g}_{\ast}\in l^{2}_{{\rm
L^{\ast}}}$, and the inner products {\rm(\ref{EQ: InnerProd SpSeq-s}),
(\ref{EQ: InnerProd SpSeq-s_2})} take the form
$$
(\widehat{f},\ \widehat{g})_{l^{2}_{{\rm L}}}=\sum_{\xi\in\ind}\widehat{f}(\xi)\ \overline{\widehat{g}_{\ast}(\xi)}
$$
and
$$
(\widehat{f}_{\ast},\ \widehat{g}_{\ast})_{l^{2}_{{\rm
L^{\ast}}}}=\sum_{\xi\in\ind}\widehat{f}_{\ast}(\xi)\
\overline{\widehat{g}(\xi)}.
$$
In particular, we have
$$
\overline{(\widehat{f},\ \widehat{g})_{l^{2}_{{\rm L}}}}=
(\widehat{g}_{\ast},\ \widehat{f}_{\ast})_{l^{2}_{{\rm
L^{\ast}}}}.
$$
The Parseval identity takes the form
\begin{equation}\label{Parseval}
(f,g)_{L^{2}}=(\widehat{f},\widehat{g})_{l^{2}_{{\rm
L}}}=\sum_{\xi\in\ind}\widehat{f}(\xi)\ \overline{\widehat{g}_{\ast}(\xi)}.
\end{equation}
Furthermore, for any $f\in L^{2}(\Omega)$, we have
$\widehat{f}\in l^{2}_{{\rm L}}$, $\widehat{f}_{\ast}\in l^{2}_{{\rm
L^{\ast}}}$, and
\begin{equation}
\label{Planch} \|f\|_{L^{2}}=\|\widehat{f}\|_{l^{2}_{{\rm
L}}}=\|\widehat{f}_{\ast}\|_{l^{2}_{{\rm L^{\ast}}}}.
\end{equation}
\end{prop}

Now we introduce Sobolev spaces generated by the operator ${\rm L}$:

\begin{defn}[Sobolev spaces $\mathcal H^{s}_{{\rm L}}(\Omega)$] \label{SobSp}
For $f\in\mathcal D'_{{\rm L}}(\Omega)\cap \mathcal D'_{{\rm L}^{*}}(\Omega)$ and $s\in\mathbb R$, we say that
$$f\in\mathcal H^{s}_{{\rm L}}(\Omega)\; \textrm{ if and only if }\;
\langle\xi\rangle^{s}\widehat{f}(\xi)\in l^{2}_{{\rm L}}.$$
We define the norm on $\mathcal H^{s}_{{\rm L}}(\Omega)$ by
\begin{equation}\label{SobNorm}
\|f\|_{\mathcal H^{s}_{{\rm
L}}(\Omega)}:=\left(\sum_{\xi\in\ind}
\langle\xi\rangle^{2s}\widehat{f}(\xi)\overline{\widehat{f}_{\ast}(\xi)}\right)^{1/2}.
\end{equation}
The Sobolev space $\mathcal H^{s}_{{\rm L}}(\Omega)$ is then the
space of ${\rm L}$-distributions $f$ for which we have
$\|f\|_{\mathcal H^{s}_{{\rm L}}(\Omega)}<\infty$. Similarly,
we can define the
space $\mathcal H^{s}_{{\rm L^{\ast}}}(\Omega)$ by the
condition
\begin{equation}\label{SobNorm2}
\|f\|_{\mathcal H^{s}_{{\rm
L^{\ast}}}(\Omega)}:=\left(\sum_{\xi\in\ind}\langle\xi\rangle^{2s}\widehat{f}_{\ast}(\xi)\overline{\widehat{f}(\xi)}\right)^{1/2}<\infty.
\end{equation}
\end{defn}
We note that the expressions in \eqref{SobNorm} and
\eqref{SobNorm2} are well-defined since the sum
$$
\sum_{\xi\in\ind}
\langle\xi\rangle^{2s}\widehat{f}(\xi)\overline{\widehat{f}_{\ast}(\xi)}=
(\langle\xi\rangle^{s}\widehat{f}(\xi),\langle\xi\rangle^{s}\widehat{f}(\xi))_{l^{2}_{\rm L}}\geq 0
$$
is real and non-negative.
Consequently, since we can write the sum in \eqref{SobNorm2} as the
complex conjugate of that in  \eqref{SobNorm}, and with both being real,
we see that the spaces $\mathcal H^{s}_{{\rm L}}(\Omega)$ and
$\mathcal H^{s}_{{\rm L^{\ast}}}(\Omega)$ coincide as sets. Moreover, we have

\begin{prop}\label{SobHilSpace}
For every $s\in\mathbb R$, the Sobolev space
$\mathcal H^{s}_{{\rm L}}(\Omega)$ is a Hilbert space with the
inner product
$$
(f,\ g)_{\mathcal H^{s}_{{\rm L}}(\Omega)}:=\sum_{\xi\in\ind
}\langle\xi\rangle^{2s}\widehat{f}(\xi)\overline{\widehat{g}_{\ast}(\xi)}.
$$
Similarly,
the Sobolev space
$\mathcal H^{s}_{{\rm L^{\ast}}}(\Omega)$ is a Hilbert space with
the inner product
$$
(f,\ g)_{\mathcal H^{s}_{{\rm
L^{\ast}}}(\Omega)}:=\sum_{\xi\in\ind}\langle\xi\rangle^{2s}\widehat{f}_{\ast}(\xi)\overline{\widehat{g}(\xi)}.
$$
For every $s\in\mathbb R$, the Sobolev spaces $\mathcal
H^{s}(\Omega)$, $\mathcal H^{s}_{{\rm L}}(\Omega)$,
and $\mathcal H^{s}_{{\rm L}^*}(\Omega)$ are
isometrically isomorphic.
\end{prop}

\subsection{Spaces $l^{p}({\rm L})$ and $l^{p}({\rm L}^*)$}
\label{SEC:lp}

In this subsection we describe the $p$-Lebesgue versions of the
spaces of Fourier coefficients. These spaces can be considered as
the extension of the usual $l^p$ spaces on the discrete set $\ind$
adapted to the fact that we are dealing with biorthogonal systems.

\begin{defn}
Thus, we introduce the spaces $l^{p}_{\rm L}=l^{p}({\rm L})$ as
the spaces of all $a\in\mathcal S'(\ind)$ such that
\begin{equation}\label{EQ:norm1}
\|a\|_{l^{p}({\rm L})}:=\left(\sum_{\xi\in\ind}| a(\xi)|^{p}
\|u_{\xi}\|^{2-p}_{L^{\infty}(\Omega)} \right)^{1/p}<\infty,\quad
\textrm{ for }\; 1\leq p\leq2,
\end{equation}
and
\begin{equation}\label{EQ:norm2}
\|a\|_{l^{p}({\rm L})}:=\left(\sum_{\xi\in\ind}| a(\xi)|^{p}
\|v_{\xi}\|^{2-p}_{L^{\infty}(\Omega)} \right)^{1/p}<\infty,\quad
\textrm{ for }\; 2\leq p<\infty,
\end{equation}
and, for $p=\infty$,
$$
\|a\|_{l^{\infty}({\rm L})}:=\sup_{\xi\in\ind}\left( |a(\xi)|\cdot
\|v_{\xi}\|^{-1}_{L^{\infty}(\Omega)}\right)<\infty.
$$
\end{defn}

\begin{rem}\label{REM:lps}
We note that in the case of $p=2$, we have already defined the
space $l^{2}({\rm L})$ by the norm \eqref{EQ:l2norm}. There is no
problem with this since the norms
\eqref{EQ:norm1}-\eqref{EQ:norm2} with $p=2$ are equivalent to
that in \eqref{EQ:l2norm}. Indeed, by Lemma \ref{LEM: FTl2} the
first one gives a homeomorphism between $l^{p}({\rm L})$ with
$p=2$ just defined and $L^{2}(\Omega)$ while the space $l^{2}({\rm
L})$ defined by \eqref{EQ:l2norm} is isometrically isomorphic to
$L^{2}(\Omega)$ by the Plancherel identity in Proposition
\ref{PlanchId}. Therefore, both norms lead to the same space which
we denote by $l^{2}({\rm L})$. The norms
\eqref{EQ:norm1}-\eqref{EQ:norm2} with $p=2$ and the one in
\eqref{EQ:l2norm} are equivalent, but there are advantages in
using both of them. Thus, the norms
\eqref{EQ:norm1}-\eqref{EQ:norm2} allow us to view $l^{2}({\rm
L})$ as a member of the scale of spaces $l^{p}({\rm L})$ for
$1\leq p\leq \infty$ with subsequent functional analytic
properties, while the norm \eqref{EQ:l2norm} is the one for which
the Plancherel identity \eqref{Planch} holds.
\end{rem}

Analogously, we also introduce spaces $l^{p}_{{\rm
L^{\ast}}}=l^{p}({\rm L^{\ast}})$ as the spaces of all
$b\in\mathcal S'(\ind)$ such that the following norms are finite:
$$
\|b\|_{l^{p}({\rm L^{\ast}})}=\left(\sum_{\xi\in\ind}| b(\xi)|^{p}
\|v_{\xi}\|^{2-p}_{L^{\infty}(\Omega)} \right)^{1/p},\quad
\textrm{ for }\; 1\leq p\leq2,
$$
$$
\|b\|_{l^{p}({\rm L^{\ast}})}=\left(\sum_{\xi\in\ind}| b(\xi)|^{p}
\|u_{\xi}\|^{2-p}_{L^{\infty}(\Omega)} \right)^{1/p},\quad
\textrm{ for }\; 2\leq p<\infty,
$$
$$
\|b\|_{l^{\infty}({\rm
L^{\ast}})}=\sup_{\xi\in\ind}\left(|b(\xi)|\cdot
\|u_{\xi}\|^{-1}_{L^{\infty}(\Omega)}\right).
$$
Before we discuss several basic properties of the spaces
$l^{p}({\rm L})$, we recall a useful fact on the interpolation of
weighted spaces from Bergh and L\"ofstr\"om \cite[Theorem
5.5.1]{Bergh-Lofstrom:BOOK-Interpolation-spaces}:

\begin{theorem}[Interpolation of weighted spaces] \label{TH: IWS}
Let us write $d\mu_{0}(x)=\omega_{0}(x)d\mu(x),$
$d\mu_{1}(x)=\omega_{1}(x)d\mu(x),$ and write
$L^{p}(\omega)=L^{p}(\omega d\mu)$ for the weight $\omega$.
Suppose that $0<p_{0}, p_{1}<\infty$. Then
$$
(L^{p_{0}}(\omega_{0}), L^{p_{1}}(\omega_{1}))_{\theta,
p}=L^{p}(\omega),
$$
where $0<\theta<1$,
$\frac{1}{p}=\frac{1-\theta}{p_{0}}+\frac{\theta}{p_{1}}$, and
$\omega=\omega_{0}^{\frac{p(1-\theta)}{p_{0}}}\omega_{1}^{\frac{p\theta}{p_{1}}}$.
\end{theorem}

From this it is easy to check that we obtain:

\begin{corollary}[Interpolation of $l^{p}({\rm L})$ and $l^{p}({\rm
L}^{\ast})$ spaces] For $1\leq p\leq2$, we have
$$
(l^{1}({\rm L}), l^{2}({\rm L}))_{\theta,p}=l^{p}({\rm L}),
$$
$$
(l^{1}({\rm L}^{\ast}), l^{2}({\rm
L}^{\ast}))_{\theta,p}=l^{p}({\rm L}^{\ast}),
$$
where $0<\theta<1$ and $p=\frac{2}{2-\theta}$.
\end{corollary}

\begin{rem}
The reason that the interpolation above is restricted to $1\leq
p\leq2$ is that the definition of $l^p$-spaces changes when we
pass $p=2$, in the sense that we use different families of
biorthogonal systems $u_\xi$ and $v_\xi$ for $p<2$ and for $p>2$.
We note that if ${\rm L}={\rm
L}^*$ is self-adjoint, so that we can take $u_\xi=v_\xi$
for all $\xi\in\ind$, then the scales $l^{p}({\rm L})$ and
$l^{p}({\rm L}^{\ast})$ coincide and satisfy interpolation
properties for all $1\leq p<\infty$.
\end{rem}

Using these interpolation properties we can establish further
properties of the Fourier transform and its inverse:

\begin{theorem}[Hausdorff-Young inequality] \label{TH: HY}
Let $1\leq p\leq2$ and $\frac{1}{p}+\frac{1}{p'}=1$. There is a
constant $C_{p}\geq 1$ such that for all $f\in L^{p}(\Omega)$ and
$a\in l^{p}({\rm L})$ we have
\begin{equation}\label{EQ:HY}
\|\widehat{f}\|_{l^{p'}({\rm L})}\leq
C_{p}\|f\|_{L^{p}(\Omega)}\quad \textrm{ and }\quad \|\mathcal
F_{{\rm L}}^{-1}a\|_{L^{p'}(\Omega)}\leq C_{p}\|a\|_{l^{p}({\rm
L})}.
\end{equation}
Similarly, we also have
\begin{equation}\label{EQ:HYast}
\|\widehat{f}_*\|_{l^{p'}({\rm L}^*)}\leq
C_{p}\|f\|_{L^{p}(\Omega)}\quad \textrm{ and } \quad \|\mathcal
F_{{\rm L}^{*}}^{-1}b\|_{L^{p'}(\Omega)}\leq
C_{p}\|b\|_{l^{p}({\rm L}^*)},
\end{equation}
for all $b\in l^{p}({\rm L}^*)$.
\end{theorem}

It follows from the proof that if ${\rm L}$ is
self-adjoint, then the $l^{2}_{L}$-norms discussed in Remark
\ref{REM:lps} coincide, and so we can put $C_{p}=1$ in
inequalities \eqref{EQ:HY} and \eqref{EQ:HYast}. If ${\rm
L}$ is not self-adjoint, $C_{p}$ may in principle depend
on ${\rm L}$ and its domain through constants from inequalities in
Lemma \ref{LEM: FTl2}.

We now turn to the duality between spaces $l^{p}({\rm L})$ and
$l^{q}({\rm L}^{\ast})$:

\begin{theorem}[Duality of $l^{p}({\rm L})$ and $l^{p'}({\rm L}^{\ast})$] \label{TH:Duality lp}
Let $1\leq p<\infty$ and $\frac{1}{p}+\frac{1}{p'}=1$. Then
$$\left(l^{p}({\rm
L})\right)'=l^{p'}({\rm L}^{\ast}) \quad \textrm{ and }\quad
\left(l^{p}({\rm L}^{\ast})\right)'=l^{p'}({\rm L}).$$
\end{theorem}

\subsection{Schwartz' kernel theorem} \label{SEC:Schwartz}

In our case the Schwartz kernel theorem is also valid and here we
will briefly discuss it. So, from now on we will make the
following:

\begin{assump}
\label{Assumption_4}
Assume that the number
$s_0\in\mathbb R$ is such that we have
$$\sum_{\xi\in\ind} \langle\xi\rangle^{-s_0}<\infty.$$
\end{assump}
Recalling the operator ${\rm L}^{\circ}$ in \eqref{EQ:Lo-def}
the Assumption \ref{Assumption_4} is equivalent to assuming that
the operator $({\rm I}+{\rm L^\circ L})^{-\frac{s_0}{4m}}$ is Hilbert-Schmidt on $L^2(\Omega)$.
Indeed, recalling the definition of $\langle\xi\rangle$ in \eqref{EQ:angle},
namely that $\langle\xi\rangle$ are the eigenvalues of $({\rm I}+{\rm L^\circ L})^{-\frac{s_0}{2m}}$,
the condition
that the operator $({\rm I}+{\rm L^\circ L})^{-\frac{s_0}{4m}}$ is Hilbert-Schmidt is equivalent to
the condition that
\begin{equation}\label{EQ:HS-conv}
\|({\rm I}+{\rm L^\circ L})^{-\frac{s_0}{4m}}\|_{\tt HS}^2\cong \sum_{\xi\in\ind}
\langle\xi\rangle^{-s_0}<\infty.
\end{equation}
If L is elliptic, we may expect that we can take any $s_0>n$ but this depends on the
domain.
The order $s_0$ will enter the regularity properties of the Schwartz kernels.

\medskip
We will use the notations
$$
C^{\infty}_{{\rm L}}(\overline{\Omega}\times \overline{\Omega}):=
C^{\infty}_{{\rm L}}(\overline{\Omega})\bar\otimes C^{\infty}_{{\rm L}}(\overline{\Omega})
$$
and
$$
C^{\infty}_{{\rm L^{\ast}}}(\overline{\Omega}\times \overline{\Omega}):=
C^{\infty}_{{\rm L^{\ast}}}(\overline{\Omega})\bar\otimes C^{\infty}_{{\rm L^{\ast}}}(\overline{\Omega})
$$
with the Fr\'echet topologies given by the family of tensor norms
\begin{equation}\label{EQ:L-top-Tensor}
\|\varphi\otimes\psi\|_{C^{k}_{{\rm L}}(\overline{\Omega}\times \overline{\Omega})}:=\max_{j+l\leq k}
\|{\rm L}^{j}\varphi\|_{L^2(\Omega)}\|{\rm L}^{l}\psi\|_{L^2(\Omega)}, \quad k\in\mathbb N_0,
\; \varphi, \psi\in C_{{\rm L}}^{\infty}(\overline{\Omega})
\end{equation}
and
\begin{equation}\label{EQ:L-top-Tensor-2}
\|\varphi\otimes\psi\|_{C^{k}_{{\rm L}^*}(\overline{\Omega}\times \overline{\Omega})}:=\max_{j+l\leq k}
\|({\rm L}^*)^{j}\varphi\|_{L^2(\Omega)}\|({\rm L}^*)^{l}\psi\|_{L^2(\Omega)}
\end{equation}
for all $k\in\mathbb N_0,
\; \varphi, \psi\in C_{{\rm L}^*}^{\infty}(\overline{\Omega}),$ respectively, and for the corresponding dual spaces we write
$$\mathcal D'_{{\rm L}}(\Omega\times\Omega):=
\left(C^{\infty}_{{\rm L}}(\overline{\Omega}\times \overline{\Omega})\right)^\prime, $$
$$\mathcal D'_{{\rm L^{\ast}}}(\Omega\times\Omega):=
\left(C^{\infty}_{{\rm L^{\ast}}}(\overline{\Omega}\times \overline{\Omega})\right)^\prime.$$

For any linear continuous operator
$$A:C^{\infty}_{{\rm L}}(\overline{\Omega})\rightarrow \mathcal D'_{{\rm L}}(\Omega)$$
there exists a kernel $K_{A}\in \mathcal D'_{{\rm
L}}(\Omega\times\Omega)$ such that for all $f\in C^{\infty}_{{\rm
L}}(\overline{\Omega})$, we can write, in the sense of distributions,
\begin{equation}\label{EQ:int1}
Af(x)=\int_{\Omega}K_{A}(x,y)f(y)dy.
\end{equation}
As usual, $K_{A}$ is called the Schwartz kernel of $A$.
For $f\in C^{\infty}_{{\rm L}}(\overline{\Omega})$, using the Fourier series formula
$$
f(y)=\sum\limits_{\eta\in\ind}\widehat{f}(\eta) u_{\eta}(y),
$$
we can also write
\begin{equation}\label{EQ:int2}
Af(x)=\sum\limits_{\eta\in\ind}\widehat{f}(\eta)\int_{\Omega}K_{A}(x,y)u_{\eta}(y)dy.
\end{equation}

Also, for any linear continuous operator
$$A:C^{\infty}_{{\rm L^{\ast}}}(\overline{\Omega})\rightarrow \mathcal D'_{{\rm L^{\ast}}}(\Omega)$$
there exists a kernel $\widetilde{K}_{A}\in \mathcal D'_{{\rm
L^{\ast}}}(\Omega\times\Omega)$ such that for all $f\in C^{\infty}_{{\rm
L^{\ast}}}(\overline{\Omega})$, we can write, in the sense of distributions,
\begin{equation}\label{EQ:int1}
Af(x)=\int_{\Omega}\widetilde{K}_{A}(x,y)f(y)dy.
\end{equation}

\section{${\rm L}$--admissible operators and ${\rm L}$-quantization}
\label{SEC:admissible operators}

In this section we describe the ${\rm L}$-quantization of the ${\rm L}$--admissible operator induced by
the operator ${\rm L}$.

\begin{defn} We say that the linear continuous operator
$$A:C^{\infty}_{{\rm L}}(\overline{\Omega})\rightarrow \mathcal D'_{{\rm L}}(\Omega)$$
belongs to the class of L--admissible operators if
$$
\sum\limits_{\eta\in\ind}u_{\eta}^{-1}(x)\, u_{\eta}(z)\,
\int_{\Omega}K_{A}(x,y)u_{\eta}(y)dy
$$
is in $\mathcal D'_{{\rm L}}(\Omega\times\Omega)$.
\end{defn}

\begin{rem}
\label{RM:L-admissible: simple case}
In the case when ${\rm L}$ is the Laplace operator with periodic boundary conditions on the torus $\mathbb T^{n}$ the class of L--admissible operators coincides with the class of all periodic pseudo-differential operators as in \cite{Ruzhansky-Turunen-JFAA-torus}.
\end{rem}

So, from now on we will assume that operators $A:C^{\infty}_{{\rm
L}}(\overline{\Omega})\rightarrow \mathcal D'_{{\rm L}}(\Omega)$ are
from the class of L--admissible operators.

\begin{rem}
\label{RM:L-admissible}
Note, that the expression
$$
u_{\eta}^{-1}(x)\,
\int_{\Omega}K_{A}(x,y)u_{\eta}(y)dy
$$
exists for any operator $A$ from the class of L--admissible operators. Moreover, it
is in $\mathcal D'_{{\rm L}}(\Omega)\otimes\mathcal S'(\ind).$
\end{rem}

Indeed, since $\sum\limits_{\eta\in\ind}u_{\eta}^{-1}(x)\, u_{\eta}(z)\,
\int_{\Omega}K_{A}(x,y)u_{\eta}(y)dy$ is in $\mathcal D'_{{\rm L}}(\Omega\times\Omega)$, by taking Fourier transform in z, we get this statement. We now define the L-symbol of an L-admissible operator.

\begin{defn}[${\rm L}$-Symbols of operators] \label{$L$--Symbols}
The ${\rm L}$-symbol of a linear continuous L--admissible
operator $$A:C^{\infty}_{{\rm L}}(\overline{\Omega})\rightarrow
\mathcal D'_{{\rm L}}(\Omega)$$ is defined by
$$\sigma_{A}(x, \xi):=u_{\xi}^{-1}(x)\,
\int_{\Omega}K_{A}(x,y)u_{\xi}(y)dy.$$
\end{defn}
This is well-defined as an element of
$\mathcal D'_{{\rm L}}(\Omega)\otimes\mathcal S'(\ind)$
in view of Remark \ref{RM:L-admissible}.

\medskip

Indeed, we have
$$
Au_{\xi}=\int_{\Omega}K_{A}(x,y)u_{\xi}(y)dy,
$$
and for $f\in C^{\infty}_{{\rm L}}(\overline{\Omega})$ from the expansion
$$
f(x)=\sum_{\xi\in\ind}\widehat{f}(\xi)u_{\xi}(x)
$$
and by the operator $A:C^{\infty}_{{\rm L}}(\overline{\Omega})\rightarrow
\mathcal D'_{{\rm L}}(\Omega)$ acting on $f$, we get
$$
Af(x)=\sum_{\xi\in\ind}\widehat{f}(\xi)Au_{\xi}(x)=\sum_{\xi\in\ind}\widehat{f}(\xi)\int_{\Omega}K_{A}(x,y)u_{\xi}(y)dy.
$$
Now, if we define
$$
u_{\xi}(x)\sigma_{A}(x, \xi):=\int_{\Omega}K_{A}(x,y)u_{\xi}(y)dy,
$$
we have the implication
\begin{equation}\label{EQ:KofA}
K_{A}(x,y)=\sum_{\xi\in\ind}u_{\xi}(x)\sigma_{A}(x, \xi)\overline{v_{\xi}(y)}.
\end{equation}
Therefore we obtain the following representation of the operator $A$ by its symbol:

\begin{theorem}[${\rm L}$--quantization] \label{QuanOper}
Let $$A:C^{\infty}_{{\rm L}}(\overline{\Omega})\rightarrow
\mathcal D'_{{\rm L}}(\Omega)$$ be a linear continuous L--admissible
operator with {\rm L}-symbol $\sigma_{A}\in\mathcal D'_{{\rm L}}(\Omega)\otimes\mathcal S'(\ind).$ Then the ${\rm L}$--quantization
\begin{equation}\label{Quantization}
Af(x)=\sum_{\xi\in\ind}
 \widehat{f}(\xi)\,\sigma_{A}(x, \xi)\,u_{\xi}(x)
\end{equation}
is true for every $f\in C^{\infty}_{{\rm L}}(\overline{\Omega})$ .
The {\rm L}-symbol $\sigma_{A}$ can be written as
\begin{equation}\label{FormSymb}
\sigma_{A}(x,\xi)=u_{\xi}^{-1}(x)(Au_{\xi})(x).
\end{equation}
\end{theorem}

\begin{corollary}\label{COR: SymFor2}
We have the following equivalent formulae for {\rm L}-symbols:
\begin{align*}
{\rm (i)}  \,\,\,\,\, &\sigma_{A}(x, \xi)=u_{\xi}^{-1}(x)(Au_{\xi})(x);\\
{\rm (ii)} \,\,\,\,\,
&\sigma_{A}(x,\xi)=u_{\xi}^{-1}(x)\int_{\Omega}K_{A}(x,y)u_{\xi}(y)dy.
\end{align*}
\end{corollary}

Similarly, we can introduce an analogous notion of the ${\rm L^{\ast}}$-quantization.

\begin{defn} We say that the continuous operator
$$A:C^{\infty}_{{\rm L^{\ast}}}(\overline{\Omega})\rightarrow \mathcal D'_{{\rm L^{\ast}}}(\Omega)$$
belongs to the class of ${\rm L^{\ast}}$--admissible operators if
$$
\sum\limits_{\eta\in\ind}v_{\eta}^{-1}(x)\, v_{\eta}(z)\,
\int_{\Omega}\widetilde{K}_{A}(x,y)v_{\eta}(y)dy
$$
is in $\mathcal D'_{{\rm L}^{\ast}}(\Omega\times\Omega)$.
\end{defn}

\begin{rem}
\label{RM:L*-admissible}
Similarly to Remark \ref{RM:L-admissible},
note that the expression
$$
v_{\xi}^{-1}(x)\,
\int_{\Omega}\widetilde{K}_{A}(x,y)v_{\xi}(y)dy
$$
exists for any operator $A$ from the class of ${\rm L^{\ast}}$--admissible operators. Moreover, it
is in $\mathcal D'_{{\rm L^{\ast}}}(\Omega)\otimes\mathcal S'(\ind).$
\end{rem}

We also can define the ${\rm L^{\ast}}$-symbol of an ${\rm L^{\ast}}$--admissible operator.

\medskip

\begin{defn} \label{$L$--Symbols2}
The ${\rm L^{\ast}}$-symbol of a linear continuous ${\rm L}^{*}$--admissible
operator $$A:C^{\infty}_{{\rm L^{\ast}}}(\overline{\Omega})\rightarrow
\mathcal D'_{{\rm L^{\ast}}}(\Omega)$$ is defined by
$$\tau_{A}(x, \xi):=v_{\xi}^{-1}(x)\,
\int_{\Omega}\widetilde{K}_{A}(x,y)v_{\xi}(y)dy.$$
\end{defn}

\medskip

Similarly to the case of L-symbols we have
$$
Av_{\xi}=\int_{\Omega}\widetilde{K}_{A}(x,y)v_{\xi}(y)dy,
$$
and for $f\in C^{\infty}_{{\rm L^{\ast}}}(\overline{\Omega})$ from the expantion
$$
f(x)=\sum_{\xi\in\ind}\widehat{f}_{\ast}(\xi)v_{\xi}(x)
$$
and by the operator $A:C^{\infty}_{{\rm L^{\ast}}}(\overline{\Omega})\rightarrow
\mathcal D'_{{\rm L^{\ast}}}(\Omega)$ acting on $f$, we get
$$
Af(x)=\sum_{\xi\in\ind}\widehat{f}_{\ast}(\xi)Av_{\xi}(x)=\sum_{\xi\in\ind}\widehat{f}_{\ast}(\xi)\int_{\Omega}\widetilde{K}_{A}(x,y)v_{\xi}(y)dy.
$$
Now, we have
$$
v_{\xi}(x)\tau_{A}(x, \xi):=\int_{\Omega}\widetilde{K}_{A}(x,y)v_{\xi}(y)dy,
$$
hence also the implication
\begin{equation}\label{EQ:KofA2}
\widetilde{K}_{A}(x,y)=\sum_{\xi\in\ind}v_{\xi}(x)\tau_{A}(x, \xi)\overline{u_{\xi}(y)}.
\end{equation}
We also record the resulting representation of the operator $A$ by its symbol:

\begin{theorem} \label{QuanOper2}
Let $$A:C^{\infty}_{{\rm L^{\ast}}}(\overline{\Omega})\rightarrow
\mathcal D'_{{\rm L^{\ast}}}(\Omega)$$ be a linear continuous ${\rm L}^{\ast}$--admissible
operator with ${\rm L}^{\ast}$-symbol $\tau_{A}\in\mathcal D'_{{\rm L^{\ast}}}(\Omega)\otimes\mathcal S'(\ind).$ Then the ${\rm L}^{\ast}$--quantization is given by
\begin{equation}\label{Quantization2}
Af(x)=\sum_{\xi\in\ind}
 \widehat{f}_{\ast}(\xi)\,\tau_{A}(x, \xi)\,v_{\xi}(x)
\end{equation}
for every $f\in C^{\infty}_{{\rm L}}(\overline{\Omega})$ .
The ${\rm L}^{\ast}$-symbol $\tau_{A}$ can be written as
\begin{equation}\label{FormSymb}
\tau_{A}(x,\xi)=v_{\xi}^{-1}(x)(Av_{\xi})(x).
\end{equation}
\end{theorem}

\begin{corollary}\label{COR: SymFor}
We have the following equivalent formulae for  ${\rm L}^{\ast}$-symbols:
\begin{align*}
{\rm (i)}  \,\,\,\,\, &\tau_{A}(x, \xi)=v_{\xi}^{-1}(x)(Av_{\xi})(x);\\
{\rm (ii)} \,\,\,\,\,
&\tau_{A}(x,\xi)=v_{\xi}^{-1}(x)\int_{\Omega}\widetilde{K}_{A}(x,y)v_{\xi}(y)dy.
\end{align*}
\end{corollary}

We now briefly describe the
notion of Fourier multipliers which is a natural name for operators with
symbols independent of $x$. In \cite{Delgado-Ruzhansky-Togmagambetov:nuclear}
the analysis of this paper is applied to investigate the spectral properties of
such operators, so we can be brief here.

\begin{defn}\label{Lfm}
Let $A:C_L^{\infty}\omp\rightarrow C_L^{\infty}\omp$
be a continuous linear operator.
We will say
 that $A$ is an $L$-Fourier multiplier if it satisfies
\[
\efel (Af)(\xi)=\sigma(\xi)\efel (f)(\xi),\; f\in C_{L}^{\infty}\omp,
\]
for some $\sigma:\ind\rightarrow\mathbb C$.
Analogously we define $L^*$-Fourier multipliers:
Let $B:C_{L^*}^{\infty}\omp\rightarrow C_{L^*}^{\infty}\omp$
be a continuous linear operator. We will say
 that $B$ is an $L^*$-Fourier multiplier if it satisfies
\[
\efela (Bf)(\xi)=\tau(\xi)\efela (f)(\xi),\, f\in C_{L^*}^{\infty}\omp,
\]
for some $\tau:\ind\rightarrow\mathbb C$.
\end{defn}

As used in \cite{Delgado-Ruzhansky-Togmagambetov:nuclear}, we have the
following simple relation between the symbols of
an operator and its adjoint.

\begin{prop}\label{admu}
The operator $A$ is an $L$-Fourier multiplier by $\sigma(\xi)$ if and only if
$A^*$ is an $L^*$-Fourier multiplier by $\overline{\sigma(\xi)}$.
\end{prop}

\section{Difference operators}
\label{SEC:differences}

In this section we discuss difference operators that will be instrumental in defining symbol
classes for the symbolic calculus of operators.

\medskip


Let $q_{j}\in C^{\infty}({\Omega}\times{\Omega})$, $j=1,\ldots,l$, be a given family
of smooth functions.
We will call the collection of $q_j$'s {\em {\rm L}-strongly admissible} if the following properties hold:
\begin{itemize}
\item The multiplication by $q_{j}(\cdot,\cdot)$
is a continuous linear mapping on
 $C^{\infty}_{{\rm L^*}}(\overline{\Omega}\times \overline{\Omega})$, for all $j=1,\ldots,l$;
\item $q_{j}(x,x)=0$ and $\nabla_{y}q_{j}(x,y)|_{y=x}\not=0$ for all $j=1,\ldots,l$ and all $x\in\Omega$;
\item
${\rm rank}(\nabla_{y}q_{1}(x,y), \ldots, \nabla_{y}q_{l}(x,y))|_{y=x}=n
$ for all $x\in\Omega$;
\item  the diagonal in $\Omega\times\Omega$ is the only set when all of
$q_j$'s vanish:
$$
\bigcap_{j=1}^l \left\{(x,y)\in\Omega\times\Omega: \, q_j(x,y)=0\right\}=\{(x,x):\, x\in\Omega\}.
$$
\end{itemize}

We note that the first property above implies that for every $x\in\Omega$, the multiplication by
$q_{j}(\cdot,\cdot)$ is also well-defined and extends to a continuous linear mapping on
$\mathcal D'_{{\rm L}}(\Omega\times\Omega)$. Also, the last property above contains the second one
but we chose to still give it explicitly for the clarity of the exposition.

The collection of $q_j$'s with the above properties generalises the notion of a strongly
admissible collection of functions for difference operators introduced in
\cite{Ruzhansky-Turunen-Wirth:JFAA} in the context of compact Lie groups.
We will use the multi-index notation
$$
q^{\alpha}(x,y):=q^{\alpha_1}_{1}(x,y)\cdots q^{\alpha_l}_{l}(x,y).
$$

\medskip

Analogously, the notion of an ${\rm L}^{*}$-strongly admissible collection suitable for the
conjugate problem is that of a family
$\widetilde{q}_{j}\in C^{\infty}({\Omega}\times{\Omega})$, $j=1,\ldots,l$, satisfying the properties:
\begin{itemize}
\item The multiplication by $\widetilde{q}_{j}(\cdot,\cdot)$
is a continuous linear mapping on
 $C^{\infty}_{{\rm L}}(\overline{\Omega}\times\overline{\Omega})$, for all $j=1,\ldots,l$;
\item $\widetilde{q}_{j}(x,x)=0$ and $\nabla_{y}\widetilde{q}_{j}(x,y)|_{y=x}\not=0$ for all $j=1,\ldots,l$ and all $x\in\Omega$;
\item $
{\rm rank}(\nabla_{y}\widetilde{q}_{1}(x,y), \ldots, \nabla_{y}\widetilde{q}_{l}(x,y))|_{y=x}=n
$ for all $x\in\Omega$;
\item  the diagonal in $\Omega\times\Omega$ is the only set when all of
$\widetilde{q}_j$'s vanish:
$$
\bigcap_{j=1}^l \left\{(x,y)\in\Omega\times\Omega: \, \widetilde{q}_j(x,y)=0\right\}=\{(x,x):\, x\in\Omega\}.
$$
\end{itemize}
We also write
$$
\widetilde{q}^{\alpha}(x,y):=\widetilde{q}^{\alpha_1}_{1}(x,y)\cdots
\widetilde{q}^{\alpha_l}_{l}(x,y).
$$

\medskip

For an operator $A:C^{\infty}_{{\rm L}}(\overline{\Omega})\rightarrow\mathcal D'_{{\rm L}}(\Omega)$ with Schwartz kernel $K_{A}$,
let us define $A_{q^{\alpha}}:C^{\infty}_{{\rm L}}(\overline{\Omega})\rightarrow\mathcal D'_{{\rm L}}(\Omega)$ as an operator with the kernel
$$
q^{\alpha}(x,y)K_{A}(x,y),
$$
We understand this formula in the sense of distributions, i.e.
$$
\langle q^{\alpha} K_{A},{\varphi}\rangle := \langle K_{A}, q^{\alpha} \varphi \rangle \quad
(K_{A}\in \mathcal D'_{{\rm L}}(\Omega\times\Omega),\;  \varphi\in C_{{\rm
L}^*}^{\infty}(\overline{\Omega}\times\overline{\Omega})).
$$
Then, we have
$$
A_{q^{\alpha}}f(x)=\int_{\Omega}q^{\alpha}(x,y)K_{A}(x,y)f(y)dy.
$$

Also analogously, for an operator $B:C^{\infty}_{{\rm L}^{*}}(\overline{\Omega})\rightarrow\mathcal D'_{{\rm L}^{*}}(\Omega)$ with Schwartz kernel $K_{B}$,
we define $B_{\widetilde{q}^{\alpha}}:C^{\infty}_{{\rm L}^{*}}(\overline{\Omega})\rightarrow\mathcal D'_{{\rm L}^{*}}(\Omega)$ as an operator with the kernel
$$
\widetilde{q}^{\alpha}(x,y)K_{B}(x,y).
$$
We understand this formula in the sense of distributions, i.e.
$$
\langle \widetilde{q}^{\alpha} K_{B},{\varphi}\rangle := \langle K_{B}, \widetilde{q}^{\alpha} \varphi \rangle \quad
(K_{B}\in \mathcal D'_{{\rm L}^*}(\Omega\times\Omega),\;  \varphi\in C_{{\rm
L}}^{\infty}(\overline{\Omega}\times\overline{\Omega})).
$$
Then, we get
$$
B_{\widetilde{q}^{\alpha}}f(x)=\int_{\Omega}\widetilde{q}^{\alpha}(x,y)K_{B}(x,y)f(y)dy.
$$

\begin{defn}\label{DEF: DifferenceOper}\label{DEF: DifferenceOper_2}
Let
$$
A:C^{\infty}_{{\rm L}}(\overline{\Omega})\rightarrow
\mathcal D'_{{\rm L}}(\Omega)
$$
be an ${\rm L}$--admissible operator with the symbol $a\in\mathcal D'_{{\rm L}}(\Omega)\otimes\mathcal S'(\ind)$ and with the Schwartz kernel $K_{A}\in\mathcal D'_{{\rm L}}(\Omega\times\Omega)$. Then we define the difference operator
$$
\Delta_{q}^{\alpha}:\mathcal D'_{{\rm L}}(\Omega)\otimes\mathcal S'(\ind)\rightarrow\mathcal D'_{{\rm L}}(\Omega)\otimes\mathcal S'(\ind)
$$
acting on ${\rm L}$--symbols by
\begin{align*}
\Delta_{q}^{\alpha}\,a(x,\xi) & := u_{\xi}^{-1}(x)
\int_{\Omega} K_{A_{q^{\alpha}}}(x,y)u_{\xi}(y)dy {}\\
& = u_{\xi}^{-1}(x)
\int_{\Omega} q^{\alpha}(x,y)K_{A}(x,y)u_{\xi}(y)dy,
\end{align*}
where $K_{A_{q^{\alpha}}}\in\mathcal D'_{{\rm L}}(\Omega\times\Omega)$ is the Schwartz kernel of the ${\rm L}$--admissible operator $A_{q^{\alpha}}:C^{\infty}_{{\rm L}}(\overline{\Omega})\rightarrow
\mathcal D'_{{\rm L}}(\Omega).$

\medskip

Analogously, for the ${\rm L^{\ast}}$--admissible operator
$$
B:C^{\infty}_{{\rm L^{\ast}}}(\overline{\Omega})\rightarrow
\mathcal D'_{{\rm L^{\ast}}}(\Omega)
$$
with the symbol $b\in\mathcal D'_{{\rm L^{\ast}}}(\Omega)\otimes\mathcal S'(\ind)$ and with the Schwartz kernel $\widetilde{K}_{B}\in\mathcal D'_{{\rm L^{\ast}}}(\Omega\times\Omega)$
we define the difference operator
$$
\widetilde{\Delta}_{q}^{\alpha}:\mathcal D'_{{\rm L^{\ast}}}(\Omega)\otimes\mathcal S'(\ind)\rightarrow\mathcal D'_{{\rm L^{\ast}}}(\Omega)\otimes\mathcal S'(\ind)
$$
acting on ${\rm L^{\ast}}$--symbols by
\begin{align*}
\widetilde{\Delta}_{q}^{\alpha}\,b(x,\xi) & := v_{\xi}^{-1}(x)
\int_{\Omega} \widetilde{K}_{B_{q^{\alpha}}}(x,y)v_{\xi}(y)dy {}\\
& = v_{\xi}^{-1}(x)
\int_{\Omega} \widetilde{q}^{\alpha}(x,y)\widetilde{K}_{B}(x,y)v_{\xi}(y)dy,
\end{align*}
where $K_{B_{q^{\alpha}}}\in\mathcal D'_{{\rm L^{\ast}}}(\Omega\times\Omega)$ is the Schwartz kernel of the ${\rm L^{\ast}}$--admissible operator $B_{q^{\alpha}}:C^{\infty}_{{\rm L^{\ast}}}(\overline{\Omega})\rightarrow
\mathcal D'_{{\rm L^{\ast}}}(\Omega).$
\end{defn}

\medskip

We now record the Taylor expansion formula with respect to a family of $q_j$'s,
which follows from expansions of functions $g$ and
$q^{\alpha}(e,\cdot)$ by the common Taylor series:

\begin{prop}\label{TaylorExp}
Any smooth function $g\in C^{\infty}({\Omega})$ can be
approximated by Taylor polynomial type expansions, i.e. for any $e\in\Omega$, we have
$$g(x)=\sum_{|\alpha|<
N}\frac{1}{\alpha!}D^{(\alpha)}_{x}g(x)|_{x=e}\, q^{\alpha}(e,x)+\sum_{|\alpha|=
N}\frac{1}{\alpha!}q^{\alpha}(e,x)g_{N}(x)
$$
\begin{equation}
\sim\sum_{\alpha\geq
0}\frac{1}{\alpha!}D^{(\alpha)}_{x}g(x)|_{x=e}\, q^{\alpha}(e,x)
\label{TaylorExpFormula}
\end{equation}
in a neighborhood of $e\in\Omega$, where $g_{N}\in
C^{\infty}({\Omega})$ and
$D^{(\alpha)}_{x}g(x)|_{x=e}$ can be found from the recurrent formulae:
$D^{(0,\cdots,0)}_{x}:=I$ and for $\alpha\in\mathbb N_0^l$,
$$
\mathsf
\partial^{\beta}_{x}g(x)|_{x=e}=\sum_{|\alpha|\leq|\beta|}\frac{1}{\alpha!}
\left[\mathsf
\partial^{\beta}_{x}q^{\alpha}(e,x)\right]\Big|_{x=e}D^{(\alpha)}_{x}g(x)|_{x=e},
$$
where $\beta=(\beta_1, \ldots, \beta_n)$ and
$
\partial^{\beta}_{x}=\frac{\partial^{\beta_{1}}}{\partial x_{1}^{\beta_{1}}}\cdots
\frac{\partial^{\beta_{n}}}{\partial x_{n}^{\beta_{n}}}.
$
\end{prop}

Analogously, any function $g\in C^{\infty}({\Omega})$
can be approximated by Taylor polynomial type expansions corresponding to
the adjoint problem, i.e. we
have
$$g(x)=\sum_{|\alpha|<
N}\frac{1}{\alpha!}\widetilde{D}^{(\alpha)}_{x}g(x)|_{x=e}\, \widetilde{q}^{\alpha}(e,x)+\sum_{|\alpha|=
N}\frac{1}{\alpha!}\widetilde{q}^{\alpha}(e,x)g_{N}(x)
$$
\begin{equation}
\sim\sum_{\alpha\geq
0}\frac{1}{\alpha!}\widetilde{D}^{(\alpha)}_{x}g(x)|_{x=e}\, \widetilde{q}^{\alpha}(e,x)
\label{TaylorExpFormula}
\end{equation}
in a neighborhood of $e\in\Omega$, where $g_{N}\in
C^{\infty}({\Omega})$ and
$\widetilde{D}^{(\alpha)}_{x}g(x)|_{x=e}$ are found from the
recurrent formula: $\widetilde{D}^{(0,\cdots,0)}:=I$ and for
$\alpha\in\mathbb N_{0}^{l}$,
$$
\partial^{\beta}_{x}g(x)|_{x=e}=\sum_{|\alpha|\leq|\beta|}\frac{1}{\alpha!}
\left[
\partial^{k}_{x}\widetilde{q}^{\alpha}(e,x)\right]\Big|_{x=e}\widetilde{D}^{(\alpha)}_{x}g(x)|_{x=e},
$$
where $\beta=(\beta_1, \ldots, \beta_n)$, and $\partial^{\beta}$ is defined as in
Proposition \ref{TaylorExp}.

It can be seen that operators $D^{(\alpha)}$ and
$\widetilde{D}^{(\alpha)}$ are differential operators of order
$|\alpha|$. We will understand them  in distributions sense, i.e. for the  L--admissible (${\rm L}^{\ast}$--admissible)
operator $A:C^{\infty}_{{\rm L}}(\overline{\Omega})\rightarrow
\mathcal D'_{{\rm L}}(\Omega)$ ($B:C^{\infty}_{{\rm L}^{\ast}}(\overline{\Omega})\rightarrow
\mathcal D'_{{\rm L}^{\ast}}(\Omega)$) define the operator $D^{(\alpha)}A$ ($\widetilde{D}^{(\alpha)}B$) as an operator with the Schwartz kernel $D^{(\alpha)}_{x}K_{A}(x,y)$ ($\widetilde{D}^{(\alpha)}_{x}K_{B}(x,y)$). Then we can act on L--symbols (${\rm L}^{\ast}$--symbols) by $D^{(\alpha)}$ ($\widetilde{D}^{(\alpha)}$).

\section{Symbolic calculus}
\label{SEC:differences2}

Using such difference operators and derivatives $D^{(\alpha)}$ from
Proposition \ref{TaylorExp}
we can now define classes of symbols.

\begin{defn}[Symbol class $S^m_{\rho,\delta}(\overline{\Omega}\times\ind)$]\label{DEF: SymClass}
Let $m\in\mathbb R$, $0\leq\delta,\rho\leq 1$. Let
$$
A:C^{\infty}_{{\rm L}}(\overline{\Omega})\rightarrow
\mathcal D'_{{\rm L}}(\Omega)
$$
be an ${\rm L}$--admissible operator with the symbol $a\in\mathcal D'_{{\rm L}}(\Omega)\otimes\mathcal S'(\ind)$ and with the Schwartz kernel $K_{A}\in\mathcal D'_{{\rm L}}(\Omega\times\Omega)$. Then the ${\rm L}$-symbol class
$S^m_{\rho,\delta}(\overline{\Omega}\times\ind)$ consists of
such  symbols $a(x,\xi)$ which are smooth in $x$ for all
$\xi\in\ind$, and which satisfy
\begin{equation}\label{EQ:symbol-class}
  \left|\Delta_{q}^\alpha D^{(\beta)}_{x} a(x,\xi) \right|
        \leq C_{a\alpha\beta m}
                \ \langle\xi\rangle^{m-\rho|\alpha|+\delta|\beta|}
\end{equation}
for all $x\in\overline{\Omega}$, for all $\alpha,\beta\geq 0$, and for all $\xi\in\ind$.
Here we understand $D^{(\beta)}_{x} a(x,\xi)$ as the symbol of the operator $D^{(\beta)}_{x}A$, where the operators $D^{(\beta)}_{x}$ are defined in Proposition
\ref{TaylorExp}. We will often denote them simply by $D^{(\beta)}$.

The class $S^m_{1,0}(\overline{\Omega}\times\ind)$ will be often
denoted by writing simply $S^m(\overline{\Omega}\times\ind)$.
In \eqref{EQ:symbol-class}, we assume that the inequality is satisfied for $x\in\Omega$ and
it extends to the closure $\overline\Omega$.
Furthermore, we define
$$
S^{\infty}_{\rho,\delta}(\overline{\Omega}\times\ind):=\bigcup\limits_{m\in\mathbb
R}S^{m}_{\rho,\delta}(\overline{\Omega}\times\ind)
$$
and
$$
S^{-\infty}(\overline{\Omega}\times\ind):=\bigcap\limits_{m\in\mathbb
R}S^{m}(\overline{\Omega}\times\ind).
$$
When we have two L-strongly admissible collections, expressing one in terms of
the other similarly to Proposition \ref{TaylorExp} and arguing similarly to
\cite{Ruzhansky-Turunen-Wirth:JFAA}, we can convince ourselves that for $\rho>\delta$ the
definition of the symbol class does not depend on the choice of an
 L-strongly admissible collection.

Analogously, we define for the ${\rm L}^*$--admissible operator
$$
B:C^{\infty}_{{\rm L}^*}(\overline{\Omega})\rightarrow
\mathcal D'_{{\rm L}^*}(\Omega)
$$
with the symbol $b\in\mathcal D'_{{\rm L}^*}(\Omega)\otimes\mathcal S'(\ind)$ and with the Schwartz kernel $\widetilde{K}_{B}\in\mathcal D'_{{\rm L}^*}(\Omega\times\Omega)$
the ${\rm L^{\ast}}$-symbol class
$\widetilde{S}^m_{\rho,\delta}(\overline{\Omega}\times\ind)$
as the space  of such symbols $b(x,\xi)$ which are smooth in $x$ for
all $\xi\in\ind$, and which satisfy
\begin{equation*}
  \left|\widetilde{\Delta}_{(x)}^\alpha \widetilde{D}^{(\beta)} b(x,\xi) \right|
        \leq C_{a\alpha\beta m}
                \ \langle\xi\rangle^{m-\rho|\alpha|+\delta|\beta|}
\end{equation*}
for all $x\in\overline{\Omega}$, for all $\alpha,\beta\geq 0$, and for all $\xi\in\ind$. Here we understand $\widetilde{D}^{(\beta)} b(x,\xi)$ as the symbol of the operator $\widetilde{D}^{(\beta)} B$.
Similarly, we can define classes
$\widetilde{S}^{\infty}_{\rho,\delta}(\overline{\Omega}\times\ind)$
and $\widetilde{S}^{-\infty}(\overline{\Omega}\times\ind)$.
\end{defn}

If $a\in S^m_{\rho,\delta}(\overline{\Omega}\times\ind)$, it is convenient to
denote by $a(X,D)={\rm Op_L}(a)$  the corresponding ${\rm
L}$-pseudo-differential operator defined by
\begin{equation}\label{EQ: L-tor-pseudo-def}
  {\rm Op_L}(a)f(x)=a(X,D)f(x):=\sum_{\xi\in\ind} u_{\xi}(x)\ a(x,\xi)\widehat{f}(\xi).
\end{equation}
The set of operators ${\rm Op_L}(a)$ of the form
(\ref{EQ: L-tor-pseudo-def}) with $a\in
S^m_{\rho,\delta}(\overline{\Omega}\times\ind)$ will be denoted by
${\rm Op_L}(S^m_{\rho,\delta} (\overline{\Omega}\times\ind))$, or by
$\Psi^m_{\rho,\delta} (\overline{\Omega}\times\ind)$. If an
operator $A$ satisfies $A\in{\rm
Op_L}(S^m_{\rho,\delta}(\overline{\Omega}\times\ind))$, we denote
its ${\rm L}$-symbol by $\sigma_{A}=\sigma_{A}(x, \xi), \,\,
x\in\overline{\Omega}, \, \xi\in\ind$. Naturally,
$\sigma_{a(X,D)}(x,\xi)=a(x,\xi)$.

Analogously, if $a\in
\widetilde{S}^m_{\rho,\delta}(\overline{\Omega}\times\ind)$,
we denote by $a(X,D)={\rm Op_{L^*}}(a)$  the corresponding ${\rm
L^{\ast}}$-pseudo-differential operator defined by
\begin{equation}\label{EQ: L-tor-pseudo-def_2}
  {\rm Op_{L^*}}(a)f(x)=a(X,D)f(x):=\sum_{\xi\in\ind} v_{\xi}(x)\ a(x,\xi)\widehat{f}_{\ast}(\xi).
\end{equation}
The set of operators ${\rm Op_{L^*}}(a)$ of the form (\ref{EQ:
L-tor-pseudo-def_2}) with $a\in
\widetilde{S}^m_{\rho,\delta}(\overline{\Omega}\times\ind)$
will be denoted by ${\rm Op_{L^*}}(\widetilde{S}^m_{\rho,\delta}
(\overline{\Omega}\times\ind))$, or by
$\widetilde{\Psi}^m_{\rho,\delta} (\overline{\Omega}\times\ind)$.

\begin{rem}\label{REM: Topology of SymClass}
{\rm (Topology on $S^{m}_{\rho,
\delta}(\overline{\Omega}\times\ind)$ ($\widetilde{S}^{m}_{\rho,
\delta}(\overline{\Omega}\times\ind)$)).} The set $S^{m}_{\rho,
\delta}(\overline{\Omega}\times\ind)$ ($\widetilde{S}^{m}_{\rho,
\delta}(\overline{\Omega}\times\ind)$) of symbols has a natural
topology. Let us consider the functions $p_{\alpha\beta}^{l}:
S^{m}_{\rho,
\delta}(\overline{\Omega}\times\ind)\rightarrow\mathbb R$
($\widetilde{p}_{\alpha\beta}^{l}: \widetilde{S}^{m}_{\rho,
\delta}(\overline{\Omega}\times\ind)\rightarrow\mathbb R$) defined
by
$$
p_{\alpha\beta}^{l}(\sigma):={\rm
sup}\left[\frac{\left|\Delta_{(x)}^{\alpha}D^{(\beta)}\sigma(x,
\xi)\right|}{\langle\xi\rangle^{l-\rho|\alpha|+\delta|\beta|}}:\,\,
(x, \xi)\in\overline{\Omega}\times\ind\right]
$$
$$
\left(\widetilde{p}_{\alpha\beta}^{l}(\sigma):={\rm
sup}\left[\frac{\left|\widetilde{\Delta}_{(x)}^{\alpha}\widetilde{D}^{(\beta)}\sigma(x,
\xi)\right|}{\langle\xi\rangle^{l-\rho|\alpha|+\delta|\beta|}}:\,\,
(x, \xi)\in\overline{\Omega}\times\ind\right]\right).
$$
Now $\{p_{\alpha\beta}^{l}\}$
($\{\widetilde{p}_{\alpha\beta}^{l}\}$) is a countable family of seminorms,
and they define a
Fr\'echet topology on $S^{m}_{\rho,
\delta}(\overline{\Omega}\times\ind)$
($\widetilde{S}^{m}_{\rho, \delta}(\overline{\Omega}\times\mathbb
Z)$). Due to the bijective correspondence of ${\rm
Op_L}(S^{m}_{\rho, \delta}(\overline{\Omega}\times\ind))$ and
$S^{m}_{\rho, \delta}(\overline{\Omega}\times\ind)$ (${\rm
Op_{L^*}}(\widetilde{S}^{m}_{\rho,
\delta}(\overline{\Omega}\times\ind))$ and
$\widetilde{S}^{m}_{\rho, \delta}(\overline{\Omega}\times\mathbb
Z)$), this directly topologises also the set of operators. These spaces
are not normable, and the topologies have but a marginal role.
\end{rem}

The notion of a symbol can be naturally  extended to that of an amplitude.

\begin{defn}[${\rm L}$-amplitudes]\label{DEF: Amplitude}
The class $\mathcal A^m_{\rho,\delta}(\overline{\Omega})$ of
${\rm L}$-amplitudes consists of the functions
$a(x,y,\xi)$ which are smooth in $x$ and $y$ for all
$\xi\in\ind$, and $a(x,x,\xi)$ is an  L--symbol for some L--admissible operator and which satisfy
\begin{equation}
  \left|\Delta_{(x)}^\alpha \Delta_{(y)}^{\alpha'} D^{(\beta)}_{x} D^{(\gamma)}_{y}
        a(x,y,\xi) \right|
        \leq C_{a\alpha\alpha'\beta\gamma m}
   \ \langle\xi\rangle^{m-\rho(|\alpha|+|\alpha'|)+\delta(|\beta|+|\gamma|)}
\end{equation}
for all $x,y\in\overline{\Omega}$, for all $\alpha,\alpha',
\beta,\gamma\geq 0$, and for all $\xi\in\ind$. Such a
function $a$ will be also called an ${\rm L}$-amplitude
of order $m\in\mathbb R$ of type $(\rho,\delta)$. Formally we may
also define
$$
  ({\rm Op_L}(a)f)(x) := \sum_{\xi\in\ind} \int_{\Omega}
    u_{\xi}(x)\ \overline{v_{\xi}(y)}\ a(x,y,\xi)\ f(y)\ dy
$$
for $f\in C_{{\rm L}}^\infty(\overline{\Omega})$. Sometimes we
may denote ${\rm Op_L}(a)$ by $a(X,Y,D).$
We also write $\mathcal
A^{m}(\overline{\Omega}):=\mathcal A^{m}_{1,
0}(\overline{\Omega})$ as well as
$$
\mathcal
A^{-\infty}(\overline{\Omega}):=\bigcap\limits_{m\in\mathbb
R}\mathcal A^{m}(\overline{\Omega}) \,\,\,\, \hbox{and} \,\,\,\,
\mathcal
A^{\infty}_{\rho,\delta}(\overline{\Omega}):=\bigcup\limits_{m\in\mathbb
R}\mathcal A^{m}_{\rho,\delta}(\overline{\Omega}).
$$
\end{defn}

Clearly we can regard the ${\rm L}$-symbols as a special class of
${\rm L}$-amplitudes, namely the ones
independent of the middle argument.
Analogously, the class $\widetilde{\mathcal
A}^m_{\rho,\delta}(\overline{\Omega})$ of ${\rm
L^{\ast}}$-amplitudes consists of the functions
$a(x,y,\xi)$ which are smooth in $x$ and $y$ for all
$\xi\in\ind$, and $a(x,x,\xi)$ is an  ${\rm L}^{\ast}$--symbol for some ${\rm L}^{\ast}$--admissible operator and which satisfy
\begin{equation}
  \left|\widetilde{\Delta}_{(x)}^\alpha \widetilde{\Delta}_{(y)}^{\alpha'} \widetilde{D}^{(\beta)}_{x} \widetilde{D}^{(\gamma)}_{y}
        a(x,y,\xi) \right|
        \leq C_{a\alpha\beta\gamma m}
   \ \langle\xi\rangle^{m-\rho(|\alpha|+|\alpha'|)+\delta(|\beta|+|\gamma|)}
\end{equation}
for all $x,y\in\overline{\Omega}$, for all $\alpha, \alpha',
\beta,\gamma\geq 0$, and for all $\xi\in\ind$.
Formally we may also write
$$
  ({\rm Op_{L^*}}(a)f)(x) := \sum_{\xi\in\ind} \int_{\Omega}
    v_{\xi}(x)\ \overline{u_{\xi}(y)}\ a(x,y,\xi)\ f(y)\ dy
$$
for $f\in C_{{\rm L^{\ast}}}^\infty(\overline{\Omega})$.
We also write
$\widetilde{\mathcal
A}^{m}(\overline{\Omega}):=\widetilde{\mathcal A}^{m}_{1,
0}(\overline{\Omega})$ as well as
$
\widetilde{\mathcal
A}^{-\infty}(\overline{\Omega}):=\bigcap\limits_{m\in\mathbb
R}\widetilde{\mathcal A}^{m}(\overline{\Omega})$
and $\widetilde{\mathcal
A}^{\infty}_{\rho,\delta}(\overline{\Omega}):=\bigcup\limits_{m\in\mathbb
R}\widetilde{\mathcal A}^{m}_{\rho,\delta}(\overline{\Omega}).
$

\begin{defn}[Equivalence of amplitudes] \label{DEF: EquivAmplit}
We say that amplitudes $a, a'$ are $m(\rho,
\delta)$-equivalent $(m\in\mathbb R)$,
$a\stackrel{m,\rho,\delta}{\sim} a'$, if $a-a'\in\mathcal
A^{m}_{\rho,\delta}(\overline{\Omega})$; they are asymptotically
equivalent, $a\sim a'$ (or $a\stackrel{-\infty}{\sim} a'$ if we
need additional clarity), if $a-a'\in\mathcal
A^{-\infty}(\overline{\Omega})$. For the corresponding operators we also write
${\rm Op}(a)\stackrel{m,\rho,\delta}{\sim} {\rm Op}(a')$ and ${\rm
Op}(a)\sim {\rm Op}(a')$ (or ${\rm Op}(a)\stackrel{-\infty}{\sim}
{\rm Op}(a')$ if we need additional clarity), respectively. It is
obvious that $\stackrel{m,\rho,\delta}{\sim}$ and $\sim$ are
equivalence relations.
\end{defn}

From the algebraic point of view, we could handle the amplitudes,
symbols, and operators modulo the equivalence relation $\sim$,
because the ${\rm L}$-pseudo-differential operators form a
$\ast$-algebra with ${\rm
Op}(S^{-\infty}(\overline{\Omega}\times\ind))$ as a
subalgebra.

\vspace{3mm}

The next theorem is a prelude to asymptotic expansions, which are
the main tool in the symbolic analysis of ${\rm
L}$-pseudo-differential operators.

\begin{theorem}[Asymptotic sums of symbols] Let $(m_{j})_{j=0}^{\infty}\subset\mathbb
R$ be a sequence such that $m_{j}>m_{j+1}$, and
$m_{j}\rightarrow-\infty$ as $j\rightarrow\infty$, and
$\sigma_{j}\in
S^{m_{j}}_{\rho,\delta}(\overline{\Omega}\times\ind)$ for all
$j\in\ind$. Then there exists an ${\rm L}$-symbol $\sigma\in
S^{m_{0}}_{\rho,\delta}(\overline{\Omega}\times\ind)$ such that
for all $N\in\ind$,
$$
\sigma\stackrel{m_{N},\rho,\delta}{\sim}\sum_{j=0}^{N-1}\sigma_{j}.
$$
\end{theorem}

We will now look at the formula for the symbol of the adjoint operator.
Let $A\in {\rm Op_L}
(S^m_{\rho,\delta}(\overline{\Omega}\times\ind))$. By the
definition of the adjoint operator we have
$$
(Au_{\xi}, v_{\eta})_{L^2}=(u_{\xi}, A^{*}v_{\eta})_{L^2}
$$
or
$$
\int_{\Omega}Au_{\xi}(x)\overline{v_{\eta}(x)}dx=\int_{\Omega}u_{\xi}(x)\overline{A^{\ast}v_{\eta}(x)}dx
$$
for $\xi, \eta\in\ind$.
Plugging in the integral expressions, we get
\begin{align*}
\int_{\Omega}{\left[\int_{\Omega}K_{A}(x,y)u_{\xi}(y)dy\right]}\overline{v_{\eta}(x)}dx & =
\int_{\Omega}{u_{\xi}(x)}\overline{\left[\int_{\Omega}K_{A^{\ast}}(x,y)v_{\eta}(y)dy\right]}dx \\
& = \int_{\Omega}{u_{\xi}(y)}{\left[\int_{\Omega}\overline{K_{A^{\ast}}(y,x)}
\overline{v_{\eta}(x)}dx\right]}dy
\end{align*}
for $\xi, \eta\in\ind$, where we swapped $x$ and $y$ in the last formula.
Consequently, we get the familiar property
$$
K_{A^{\ast}}(x,y)=\overline{K_{A}(y,x)}.
$$
Now, using this and the equation (\ref{EQ:KofA}), and formula (ii) in Corollary \ref{COR: SymFor2},
and then formula (ii) in Corollary \ref{COR: SymFor} and the Taylor expansion in Proposition \ref{TaylorExp},
we can write for the ${\rm L}^*$-symbol $\tau_{A^*}$ of $A^*$ that
\begin{align*}
 \tau_{A^*}(x,\xi) & =
v_{\xi}^{-1}(x) \int_\Omega K_{A^{\ast}}(x,y)v_\xi(y) dy \\
& =
v_{\xi}^{-1}(x) \int_\Omega \overline{K_{A}(y,x)} v_\xi(y) dy \\
& =
v_{\xi}^{-1}(x)\int_\Omega \sum_{\eta\in\ind}  \overline{u_\eta(y)  \sigma_A(y,\eta)} v_\eta(x)  v_\xi(y) dy \\
& \sim v_{\xi}^{-1}(x) \int_\Omega \sum_{\eta\in\ind}  \overline{u_\eta(y)}
\sum_\alpha \frac{1}{\alpha!} \overline {D_x^{(\alpha)} \sigma_A(x,\eta)
q^\alpha(x,y)} v_\eta(x)  v_\xi(y) dy
\end{align*}
as an asymptotic sum. Formally regrouping terms for each $\alpha$, we obtain
$$
 \tau_{A^*}(x,\xi) \sim
\sum_\alpha \frac{1}{\alpha!} v_{\xi}^{-1}(x)  \int_\Omega\sum_{\eta\in\ind}  \overline {u_\eta(y) q^\alpha(x,y) D_x^{(\alpha)} \sigma_A(x,\eta)}  v_\eta(x)  v_\xi(y) dy.
$$
Using the formula (\ref{EQ:KofA2}), and
taking $$\widetilde{q}(x,y):=\overline{q(x,y)}$$
we can write this as
$$
\tau_{A^*}(x,\xi) \sim \sum_\alpha \frac{1}{\alpha!}
\widetilde \Delta_{\widetilde{q}}^\alpha D_x^{(\alpha)}\overline{\sigma_A(x,\xi)}.
$$
Making rigorous estimates for the remainder in a routine way,
and assuming in the following theorem that for every $x\in\Omega$,
the multiplication by $q_{j}(x,\cdot)$ preserves both spaces
$C_{{\rm L}}^\infty(\overline{\Omega})$ and $C_{{\rm L}^{*}}^\infty(\overline{\Omega})$,
we obtained:

\begin{theorem}[Adjoint operators]
Let $0\leq\delta<\rho\leq 1$. Let $A\in {\rm Op_L}
(S^m_{\rho,\delta}(\overline{\Omega}\times\ind))$.
Assume that the conjugate symbol
class $\widetilde{S}^{m}_{\rho,\delta}(\overline{\Omega}\times\ind)$
is defined with strongly admissible
functions $\widetilde{q}_{j}(x,y):=\overline{q_{j}(x,y)}$ which are ${\rm L}^{*}$-strongly admissible.
Then the adjoint of $A$ satisfies
$A^{\ast}\in {\rm Op_{L^*}}(\widetilde{S}^{m}_{\rho,\delta}(\overline{\Omega}\times\ind))$,
with its ${\rm L}^*$-symbol
$\tau_{A^*}\in \widetilde{S}^{m}_{\rho,\delta}(\overline{\Omega}\times\ind)$
having the asymptotic expansion
$$
\tau_{A^*}(x,\xi) \sim \sum_\alpha \frac{1}{\alpha!}
\widetilde \Delta_x^\alpha D_x^{(\alpha)}\overline{\sigma_A(x,\xi)}.
$$
\end{theorem}

We now treat symbols of the amplitude operators.

\begin{theorem}[Amplitude symbols]
Let $0\leq\delta<\rho\leq 1$ and let $a\in \mathcal
A^m_{\rho,\delta}(\overline{\Omega})$
be such that ${\rm Op_L}(a)$ is $L$-admissible.
Then there exists a unique {\rm L}-symbol
$\sigma\in S^m_{\rho,\delta}(\overline{\Omega}\times\ind)$
satisfying ${\rm Op_L}(a)={\rm Op_L}(\sigma)$, where
\begin{equation}
  \sigma(x,\xi) \sim \sum_{\alpha\geq 0 } \frac{1}{\alpha!}
        \ \Delta_{(x)}^{\alpha}
        \ D_y^{(\alpha)} a(x,y,\xi)|_{y=x}.
\end{equation}
\end{theorem}

\begin{proof} As a linear operator on $C_{{\rm
L}}^{\infty}(\overline{\Omega})$, the operator ${\rm Op_L}(a)$ possesses the
unique L-symbol $\sigma=\sigma_{{\rm Op_L}(a)}$, but at the moment we
do not yet know whether $\sigma\in
S^m_{\rho,\delta}(\overline{\Omega}\times\ind)$. By Theorem
\ref{QuanOper}  the L-symbol is computed from
$$
\sigma(x,\xi)=u_{\xi}^{-1}(x)({\rm Op_L}(a)u_{\xi})(x)
=u_{\xi}^{-1}(x)\sum_{\eta\in\ind} \int_{\Omega}
    u_{\eta}(x)\ \overline{v_{\eta}(y)}\ a(x,y,\eta)\ u_{\xi}(y) dy.
$$
Now we approximate
the function $a(x,\cdot,\eta)\in C^{\infty}(\Omega)$ by Taylor polynomial type
expansions, by using Proposition \ref{TaylorExp}, we have
\begin{align*}
\sigma(x,\xi)&\sim u_{\xi}^{-1}(x) \sum_{\alpha\geq 0}\frac{1}{\alpha!}\sum_{\eta\in\ind} \int_{\Omega} u_{\eta}(x)\ \overline{v_{\eta}(y)} q^{\alpha}(x,y) \big[
D^{(\alpha)}_{y}a(x,y,\eta)\big]_{y=x}\
u_{\xi}(y) dy
\\
&\sim \sum_{\alpha\geq 0} \frac{1}{\alpha!}
        \ \Delta_{(x)}^\alpha
        \ D_y^{(\alpha)} a(x,y,\xi)|_{y=x},
\end{align*}
Omitting a routine verification of the properties of the remainder,
this yields the statement.
\end{proof}

We now formulate the composition formula.

\begin{theorem}\label{Composition}
Let $m_{1}, m_{2}\in\mathbb R$ and $\rho>\delta\geq0$. Let $A,
B:C_{{\rm L}}^{\infty}(\overline{\Omega})\rightarrow C_{{\rm
L}}^{\infty}(\overline{\Omega})$ be linear continuous and L--admissible operators, and assume that
their {\rm L}-symbols satisfy
\begin{align*}
|\Delta_{(x)}^{\alpha}\sigma_{A}(x,\xi)|&\leq
C_{\alpha}\langle\xi\rangle^{m_{1}-\rho|\alpha|},\\
|D^{(\beta)}\sigma_{B}(x,\xi)|&\leq
C_{\beta}\langle\xi\rangle^{m_{2}+\delta|\beta|},
\end{align*}
for all $\alpha,\beta\geq 0$, uniformly in $x\in\overline{\Omega}$ and
$\xi\in\ind$.
Then
\begin{equation}
\sigma_{AB}(x,\xi)\sim\sum_{\alpha\geq 0}
\frac{1}{\alpha!}(\Delta_{(x)}^{\alpha}\sigma_{A}(x,\xi))D^{(\alpha)}\sigma_{B}(x,\xi),
\label{CompositionForm}
\end{equation}
where the asymptotic expansion means that for every $N\in\mathbb N$ we have
$$
|\sigma_{AB}(x,\xi)-\sum_{|\alpha|<N}\frac{1}{\alpha!}(\Delta_{(x)}^{\alpha}\sigma_{A}(x,\xi))D^{(\alpha)}\sigma_{B}(x,\xi)|\leq
C_{N}\langle\xi\rangle^{m_{1}+m_{2}-(\rho-\delta)N}.
$$
\end{theorem}

\begin{proof}
First, by the Schwartz kernel theorem from Subsection \ref{SEC:Schwartz}, we have
\begin{align*}
ABf(x)&=\int_{\Omega}K_{A}(x,y) (Bf)(y)dy
\\
&=\int_{\Omega}K_{A}(x,y)\Big[\int_{\Omega}K_{B}(y,z)f(z)dz\Big]dy
\\
&=\int_{\Omega}\int_{\Omega}K_{A}(x,y) K_{B}(y,z)f(z)dzdy.
\end{align*}
Hence
\begin{align*}
\sigma_{AB}(x,\xi)&=u_{\xi}^{-1}(x)(A(Bu_{\xi}))(x) \\
&=u_{\xi}^{-1}(x)\int_{\Omega}K_{A}(x,y)\Big[\int_{\Omega} K_{B}(y,z)u_{\xi}(z)dz\Big]dy
\\
&=u_{\xi}^{-1}(x)\int_{\Omega}K_{A}(x,y)u_{\xi}(y)\sigma_{B}(y,\xi)dy.
\end{align*}
Now we approximate the function $\sigma_{B}(\cdot,\xi)\in C_{{\rm
L}}^{\infty}(\overline{\Omega})$ by Taylor polynomial type
expansions. By using Proposition \ref{TaylorExp}, we get
\begin{align*}
\sigma_{AB}(x,\xi)&\sim u_{\xi}^{-1}(x)\int_{\Omega}K_{A}(x,y) \Big[\sum_{\alpha\geq 0}
\frac{1}{\alpha!}q^{\alpha}(x,y)D^{(\alpha)}_{x}\sigma_{B}(x,\xi)\Big]u_{\xi}(z)dy
\\
&=\sum_{\alpha\geq 0}
\frac{1}{\alpha!} \Big[ u_{\xi}^{-1}(x)\int_{\Omega} q^{\alpha}(x,y) K_{A}(x,y) u_{\xi}(y) dy \Big]D^{(\alpha)}_{x}\sigma_{B}(x,\xi)
\end{align*}
Using Definition \ref{DEF: DifferenceOper}, we have
\begin{align*}
\sigma_{AB}(x,\xi)\sim \sum_{\alpha\geq 0}
\frac{1}{\alpha!}(\Delta_{(x)}^{\alpha}\sigma_{A}(x,\xi))D^{(\alpha)}_{x}\sigma_{B}(x,\xi).
\end{align*}
Omitting a routine treatment of the remainder, this completes the proof.
\end{proof}

\section{On further results}

\subsection{Properties of integral kernels} \label{SEC:kernels}

We now establish some properties of Schwartz kernels of
pseudo-differential operators with symbols in the introduced
H\"ormander-type classes. In the following Theorem \ref{TH:
KernelofPDO}, let us make the  assumption on the growth of
$L^\infty$-norms of the eigenfunctions $u_\xi$. Finding estimates
for the norms $\|u_{\xi}\|_{L^{\infty}}$ in terms of the
corresponding eigenvalues of L is a challenging problem even for
self-adjoint operators L, see e.g. Sogge and Zelditch
\cite{Sogge-Zelditch:max-ef-growth-Duke} and references therein.
Thus, on tori or, more generally, on compact Lie groups, the
eigenfunctions of the Laplacian can be chosen to be uniformly
bounded. However, even for the Laplacian, on more general
manifolds, such growth depends on the geometry of the manifold. We
refer to \cite[Remark 8.9]{Delgado-Ruzhansky:invariant} for a more
thorough discussion of this topic as well as for a list of
relevant references.

\begin{theorem}[Kernel of a pseudo--differential operator] \label{TH: KernelofPDO}
Let $\mu_0$ be a constant such that there is $C>0$ such that for
all $\xi\in\ind$ we have
$$
\|u_{\xi}\|_{L^{\infty}}\leq C \langle\xi\rangle^{\mu_0}.
$$
Let $a\in S^{\mu}_{\rho,\delta}(\overline{\Omega}\times\ind)$,
$\rho>0$. Then the kernel $K(x,y)$ of the pseudo-differential
operator ${\rm Op_L}a$ satisfies
\begin{equation}\label{EQ:ests-L0}
({\rm L}^{\ast}_{y})^{k}(q^{\alpha}(x,y)K(x,y))\in L^\infty,
\end{equation}
for all
$|\alpha|>(\mu+mk+2\mu_0+s_0)/\rho$ and $x\neq y$, where $m$ is the
order from \eqref{EQ:angle} and $s_0$ is the
constant from Assumption \ref{Assumption_4}.
If ${\rm L}$ is a differential operator it follows that
\begin{equation}\label{EQ:ests-L}
|({\rm L}^{\ast}_{y})^{k}K(x,y)|\leq C_{Nk}|x-y|^{-N}
\end{equation}
for any $N>(\mu+mk+2\mu_0+s_0)/\rho$ and $x\neq y$.
\end{theorem}

\begin{proof}
By Corollary \ref{COR: SymFor2} we have
$$
u_{\xi}(x)a(x,\xi)= \int_{\Omega}K(x,y)u_{\xi}(y)dy,
$$
and from Definition \ref{DEF: DifferenceOper}  we get
\begin{align*}
u_{\xi}(x)\Delta_{(x)}^{\alpha}a(x,\xi)&=\int_{\Omega}q^{\alpha}(x,y)K(x,y)u_{\xi}(y)dy,
\end{align*}
and also
\begin{multline*}
u_\xi(x)\lambda_\xi^k \Delta_{(x)}^{\alpha} a(x,\xi)=\int_{\Omega}q^{\alpha}(x,y)K(x,y)\lambda_{\xi}^{k}u_{\xi}(y)dy
\\ =
\int_{\Omega}q^{\alpha}(x,y)K(x,y){\rm L}_{y}^{k}u_{\xi}(y)dy
=\int_{\Omega}
({\rm L}^{\ast}_{y})^{k}(q^{\alpha}(x,y)K(x,y))u_{\xi}(y)dy.
\end{multline*}
This means that $$({\rm L}^{\ast}_{y})^{k}(q^{\alpha}(x,y)K(x,y))=
{\mathcal F}_{\rm L}^{-1} (u_\xi(x)\lambda_\xi^k \Delta_{(x)}^{\alpha} a(x,\xi))(y).$$
Since it follows from assumptions that
$$\lambda_{\xi}^{k}\Delta_{(x)}^{\alpha}  a(x,\xi)\in
S^{\mu+mk-\rho|\alpha|}(\overline{\Omega}\times\ind),$$ we have
$$
\lambda_{\xi}^{k} |\Delta_{(x)}^{\alpha}a(x,\xi)|\leq
C\langle\xi\rangle^{\mu+mk-\rho|\alpha|}.
$$
We recall now the norm $$\|a(x,\cdot)\|_{l^{1}({\rm L})}=\sum_{\xi\in\ind}| a(x,\xi)|
\|u_{\xi}\|_{L^{\infty}(\Omega)}$$ from Subsection \ref{SEC:lp}.
It follows that
$$
\|u_{\xi}(x)\lambda_{\xi}^{k}\Delta_{(x)}^{\alpha}a(x,\xi)\|_{l^{1}({\rm L})}\leq
C\sum_{\xi\in\ind}\langle\xi\rangle^{\mu+mk-\rho|\alpha|}
\|u_{\xi}\|_{L^{\infty}(\Omega)}^2\leq
C\sum_{\xi\in\ind}\langle\xi\rangle^{\mu+mk-\rho|\alpha|+2\mu_0}.
$$
Consequently, if
$$|\alpha|>(\mu+mk+2\mu_0+s_0)/\rho,$$
where $s_0$ is the constant from Assumption \ref{Assumption_4},
we have that
$u_{\xi}(x)\lambda_{\xi}^{k}\Delta_{(x)}^{\alpha}a(x,\xi)$ is in $l^{1}({\rm L})$
with respect to $\xi$, and hence
$({\rm L}^{\ast}_{y})^{k}(q^{\alpha}(x,y)K(x,y))$ is in $L^\infty$ by the Hausdorff-Young
inequality in Theorem \ref{TH: HY}.
Since ${\rm L}^{\ast}_{y}$ is a differential operator for differential operators L, in this case we also have
$$
q^\alpha(x,y)({\rm L}^{\ast}_{y})^{k}K(x,y)\in
L^{\infty}(\Omega\times\Omega)
$$
for such $\alpha$.
By the properties of $q^\alpha$
it implies the statement of the theorem.
\end{proof}

In particular, if L is for example locally elliptic,
\eqref{EQ:ests-L} implies that for $x\neq y$, the kernel $K(x,y)$
is a smooth function. And, if $a\in
S^{-\infty}(\overline{\Omega}\times\ind)$, then the integral
kernel $K(x,y)$ of ${\rm Op_L}a$ is smooth in $x$ and $y$.

The singular support of $w\in\mathcal D'_{{\rm L}}(\Omega)$ is
defined as the complement of the set where $w$ coincides with a test function. Namely,
$x\notin {\rm sing\,supp}\,\, w$ if there is an open neighbourhood
$U$ of $x$ and a smooth function $f\in C_{{\rm
L}}^{\infty}(\overline{\Omega})$ such that $w(\varphi)=f(\varphi)$
for all $\varphi\in C_{{\rm L}}^{\infty}(\overline{\Omega})$ with
${\rm supp} \,\varphi\subset U$. As an immediate consequence of
Theorem \ref{TH: KernelofPDO} we obtain the information on how the
singular support is mapped by a pseudo-differential operator:

\begin{corollary}\label{COR: SingularSupp}
Let $\sigma_{A}\in
S^{\mu}_{\rho,\delta}(\overline{\Omega}\times\ind)$, $1\geq
\rho>\delta\geq 0$. Then for every $w\in\mathcal D'_{{\rm
L}}(\Omega)$ we have
\begin{equation*}\label{EQ: SingSupp}
{\rm sing\,supp}\,\, Aw\subset {\rm sing\,supp}\,\, w.
\end{equation*}
\end{corollary}
For elliptic operators, in Corollary \ref{COR:SingularSupp-ell} we
state also the inverse inclusion.

\subsection{${\rm L}$-elliptic pseudo--differential operators}
\label{SEC:elliptic}

In this subsection we discuss operators that are elliptic in the
symbol classes generated by L. For such operators we can obtain
parametrix and then also a-priori estimates by the properties of
pseudo-differential operators in, for example, Sobolev spaces,
once they are established in Section \ref{SEC:L2}, see Theorem
\ref{L2-Bs and El-ty}. Thus, from the asymptotic expansion for the
composition of pseudo-differential operators, we get an expansion
for a parametrix of an elliptic operator:

\begin{theorem}[L-ellipticity]\label{El-ty}
Let $1\geq \rho>\delta\geq 0$. Let $\sigma_A\in
S^\mu_{\rho,\delta}(\overline{\Omega}\times\ind)$ be elliptic in
the sense that there exist constants $C_0>0$ and $N_0\in\mathbb N$
such that
\begin{equation}\label{elliptic}
  |\sigma_A(x,\xi)|
  \geq C_0 \langle\xi\rangle^\mu
\end{equation}
for all $(x,\xi)\in\overline{\Omega}\times\ind$ for which $\xi\geq
N_0$; this is equivalent to assuming that there exists
$\sigma_B\in S^{-\mu}_{\rho,\delta}(\overline{\Omega}\times\ind)$
such that $I-BA,I-AB$ are in ${\rm Op_L} S^{-\infty}$.
 Let $$A \sim \sum_{j=0}^\infty A_j,$$
with $\sigma_{A_j}\in S^{\mu-(\rho-\delta)j}_{\rho,\delta}(\overline{\Omega}\times\ind)$.
Then $$B \sim \sum_{k=0}^\infty B_k, $$ where $B_k\in
S^{-\mu-(\rho-\delta)k}_{\rho,\delta}(\overline{\Omega}\times\ind)$
is such that
$$\sigma_{B_0}(x,\xi)= 1/\sigma_{A_0}(x,\xi)$$ for large enough
$\xi$, and recursively
$$
  \sigma_{B_N}(x,\xi) = \frac{-1}{\sigma_{A_0}(x,\xi)}
  \sum_{k=0}^{N-1} \sum_{j=0}^{N-k}
  \sum_{|\alpha|=N-j-k}
        \frac{1}{\alpha!} \left[
          \Delta_{(x)}^{\alpha} \sigma_{A_j}(x,\xi) \right]
        D_x^{(\alpha)} \sigma_{B_k}(x,\xi).
$$
\end{theorem}

Theorem \ref{TH: KernelofPDO} applied to the parametrix from in
Theorem \ref{El-ty}, implies the inverse inclusion to the singular
supports from Corollary \ref{COR: SingularSupp} for elliptic
operators:

\begin{corollary}\label{COR:SingularSupp-ell}
Let $1\geq \rho>\delta\geq 0$ and assume that $\sigma_{A}\in
S^{\mu}_{\rho,\delta}(\overline{\Omega}\times\ind)$ is {\rm
L}-elliptic. Then for every $w\in\mathcal D'_{{\rm L}}(\Omega)$ we
have
\begin{equation*}\label{EQ: SingSupp}
{\rm sing\,supp}\,\, Aw={\rm sing\,supp}\,\, w.
\end{equation*}
\end{corollary}

\subsection{Sobolev embedding theorem} \label{SEC:embeddings}


In this subsection we give an example of a Sobolev embedding
theorem for Sobolev spaces $\mathcal H^s_{\rm L}$ associated to L,
considered in Section \ref{SEC:Sobolev}. However, only limited
conclusions are possible in the abstract setting when no further
specifics about L are available. Now, let $C({\Omega})$ be the
Banach space under the norm
$$
\|f\|_{C({\Omega})}:=\sup\limits_{x\in{\Omega}} |f(x)|.
$$
We recall that we have a differential operator L of order $m$ with
smooth coefficients in the open set $\Omega\subset\mathbb R^n$,
and also the operator ${\rm L}^\circ$ from \eqref{EQ:Lo-def}.

\medskip
The following theorem is conditional to the local regularity
estimate \eqref{EQ:Sob-as}. It is satisfied with $\varkappa=1$ if,
for example, L is locally elliptic, i.e. elliptic in the classical
sense of $\mathbb R^n$. However, if L is for example a sum of
squares satisfying H\"ormander's commutator condition, the number
$\varkappa\geq 1$ may depend on the order to which the H\"ormander
condition is satisfied, see e.g. \cite{Garetto-Ruzhansky:sum-JDE}
in the context of compact Lie groups.

\begin{theorem}\label{TH:SETh}
Let $k$ be an integer such that $k>n/2$. Let $\varkappa$ be such
that the operators ${\rm L}$ and ${\rm L}^{\circ}$ satisfy the
inequality
\begin{equation}\label{EQ:Sob-as}
\Big\|\frac{\partial^{\alpha}f}{\partial
x^{\alpha}}\Big\|_{L^{2}(\Omega)}\leq C\Big\|({\rm I}+{\rm
L}^\circ{\rm L})^{\frac{\varkappa k}{2m}}f\Big\|_{L^{2}(\Omega)}
\end{equation}
for all $f\in C^{\infty}({\Omega})$, for all $\alpha\in\mathbb
N_0^{n}$ with $|\alpha|\leq k$. Then we have the continuous
embedding
$$
\mathcal H^{\varkappa k}_{{\rm L}}(\Omega)\hookrightarrow
C({\Omega}).
$$
\end{theorem}

The proof is similar to \cite{Ruzhansky-Tokmagambetov:IMRN} so we omit it.

\subsection{Conditions for $L^{2}$-boundedness} \label{SEC:L2}

In this subsection we will discuss what conditions on the ${\rm
L}$-symbol $a$ guarantee the $L^{2}$-boundedness of the
corresponding pseudo-differential operator ${\rm
Op_L}(a):C^{\infty}_{{\rm L}}(\overline{\Omega})\rightarrow
\mathcal D'_{{\rm L}}(\Omega)$.
The proofs of the following results are
similar to \cite{Ruzhansky-Tokmagambetov:IMRN} so we omit them.

\begin{theorem}\label{L2-Bs}
Let $k$ be an integer $>n/2$. Let
$a:\overline{\Omega}\times\ind\rightarrow\mathbb C$ be such that
\begin{equation}\label{EQ:torus-L2-pse}
  \left| \partial^{\alpha}_{x} a(x,\xi) \right| \leq
  C \quad\textrm{ for all } (x,\xi)\in\overline{\Omega}\times\ind,
\end{equation}
and all $|\alpha|\leq k$, all $x\in\Omega$ and $\xi\in\ind$. Then
the operator ${\rm Op_L}(a)$ extends to a bounded operator from
$L^{2}(\Omega)$ to $L^{2}(\Omega)$.
\end{theorem}


From a suitable adaption of the composition Theorem
\ref{Composition}, using that by Proposition \ref{TaylorExp} the
operators $\partial^{\alpha}_{x}$ and $D^{(\alpha)}_{x}$ can be
expressed in terms of each other as linear combinations with
smooth coefficients, we immediately obtain the result in Sobolev
spaces:

\begin{corollary}\label{Hs-Bs}
Let $k$ be an integer $>n/2$. Let $\mu\in\mathbb R$ and let
$a:\overline{\Omega}\times\mathbb Z\rightarrow\mathbb C$ be such
that
\begin{equation}\label{Hs-cond}
  \left| \partial^{\alpha}_{x} a(x,\xi) \right| \leq
  C \langle\xi\rangle^{\mu} \quad\textrm{ for all } (x,\xi)\in\overline{\Omega}\times\ind,
\end{equation}
and for all $\alpha$. Then operator ${\rm Op_L}(a)$ extends to a
bounded operator from $\mathcal H^{s}_{L}(\Omega)$ to $\mathcal
H^{s-\mu}_{L}(\Omega),$ for any $s\in\mathbb R.$
\end{corollary}

By using Theorem \ref{El-ty} and Corollary \ref{Hs-Bs}, we get

\begin{theorem}\label{L2-Bs and El-ty}
Let $A$ be an {\rm L}-elliptic pseudo-differential operator with {\rm
L}-symbol $\sigma_A\in S^{\mu}(\overline{\Omega}\times\ind)$,
$\mu\in\mathbb R$, and let $Au=f$ in $\Omega$, $u\in \mathcal
H^{\infty}_{{\rm L}}(\Omega)$. Then we have the estimate
$$
\|u\|_{\mathcal H^{s+\mu}_{{\rm L}}(\Omega)}\leq C_{sN}
(\|f\|_{\mathcal H^{s}_{{\rm L}}(\Omega)}+ \|u\|_{\mathcal
H^{-N}_{{\rm L}}(\Omega)}).
$$
for any $s, N\in\mathbb R$.
\end{theorem}


\begin{thebibliography}{RTW14}

\bibitem[Bar51]{bari}
N.~K. Bari.
\newblock Biorthogonal systems and bases in {H}ilbert space.
\newblock {\em Moskov. Gos. Univ. U\v cenye Zapiski Matematika},
  148(4):69--107, 1951.

\bibitem[BL76]{Bergh-Lofstrom:BOOK-Interpolation-spaces}
J.~Bergh and J.~L{{\"o}}fstr{{\"o}}m.
\newblock {\em Interpolation spaces. {A}n introduction}.
\newblock Springer-Verlag, Berlin, 1976.
\newblock Grundlehren der Mathematischen Wissenschaften, No. 223.

\bibitem[DR14a]{Delgado-Ruzhansky:invariant}
J.~Delgado and M.~Ruzhansky.
\newblock Fourier multipliers, symbols and nuclearity on compact manifolds.
\newblock {\em arXiv:1404.6479}, to appear in {\em J. Anal. Math.}

\bibitem[DRT15]{Delgado-Ruzhansky-Togmagambetov:nuclear}
J.~Delgado, M.~Ruzhansky, and N.~Tokmagambetov.
\newblock Schatten classes, nuclearity and nonharmonic analysis on compact manifolds with boundary.
\newblock {\em arXiv:1505.02261}, to appear in {\em J. Math. Pures Appl.}

\bibitem[GR15]{Garetto-Ruzhansky:sum-JDE}
C.~Garetto and M.~Ruzhansky.
\newblock Wave equation for sums of squares on compact {L}ie groups.
\newblock {\em J. Differential Equations}, 258(12):4324--4347, 2015.

\bibitem[KT14]{Kanguzhin_Tokmagambetov}
B.~Kanguzhin and N.~Tokmagambetov.
\newblock The {F}ourier transform and convolutions generated by a differential
  operator with boundary condition on a segment.
\newblock In {\em Fourier Analysis: Trends in Mathematics}, pages 235--251.
  Birkh\"auser Basel AG, Basel, 2014.

\bibitem[KTT15]{Kanguzhin_Tokmagambetov_Tulenov}
B.~Kanguzhin, N.~Tokmagambetov, and K.~Tulenov.
\newblock Pseudo-differential operators generated by a non-local boundary value
  problem.
\newblock {\em Complex Var. Elliptic Equ.}, 60(1):107--117, 2015.

\bibitem[RT16]{Ruzhansky-Tokmagambetov:IMRN}
M.~Ruzhansky and N.~Tokmagambetov.
\newblock Nonharmonic analysis of boundary value problems.
\newblock {\em Int. Math. Res. Not. IMRN}, 2016 (12), 3548--3615, 2016.

\bibitem[RT07]{RT07}
M.~Ruzhansky and V.~Turunen.
\newblock On the {F}ourier analysis of operators on the torus.
\newblock In {\em Modern trends in pseudo-differential operators}, volume 172
  of {\em Oper. Theory Adv. Appl.}, pages 87--105. Birkh{\"a}user, Basel, 2007.

\bibitem[RT09]{Ruzhansjy-Turunen:NFA}
M.~Ruzhansky and V.~Turunen.
\newblock On the toroidal quantization of periodic pseudo-differential
  operators.
\newblock {\em Numer. Funct. Anal. Optim.}, 30(9-10):1098--1124, 2009.

\bibitem[RT10a]{RT}
M.~Ruzhansky and V.~Turunen.
\newblock {\em Pseudo-differential operators and symmetries. Background
  analysis and advanced topics}, volume~2 of {\em Pseudo-Differential
  Operators. Theory and Applications}.
\newblock Birkh{\"a}user Verlag, Basel, 2010.

\bibitem[RT10b]{Ruzhansky-Turunen-JFAA-torus}
M.~Ruzhansky and V.~Turunen.
\newblock Quantization of pseudo-differential operators on the torus.
\newblock {\em J. Fourier Anal. Appl.}, 16(6):943--982, 2010.


\bibitem[RTW14]{Ruzhansky-Turunen-Wirth:JFAA}
M.~Ruzhansky, V.~Turunen, and J.~Wirth.
\newblock H{\"o}rmander class of pseudo-differential operators on compact {L}ie
  groups and global hypoellipticity.
\newblock {\em J. Fourier Anal. Appl.}, 20(3):476--499, 2014.

\bibitem[SZ02]{Sogge-Zelditch:max-ef-growth-Duke}
C.~D. Sogge and S.~Zelditch.
\newblock Riemannian manifolds with maximal eigenfunction growth.
\newblock {\em Duke Math. J.}, 114(3):387--437, 2002.



\end{thebibliography}

\end{document}